\documentclass{amsart}%
\usepackage{amsfonts}
\usepackage{amsmath}
\usepackage{amssymb}
\usepackage{graphicx}
\usepackage[latin1]{inputenc}
\usepackage[english]{babel}%
\newtheorem{theorem}{Theorem}[section]
\theoremstyle{plain}

\newtheorem{corollary}[theorem]{Corollary}

\newtheorem{definition}[theorem]{Definition}
\newtheorem{example}[theorem]{Example}

\newtheorem{lemma}[theorem]{Lemma}

\newtheorem{proposition}[theorem]{Proposition}
\newtheorem{remark}[theorem]{Remark}

\numberwithin{equation}{section}
\def\supp{\operatorname{supp}}
\def\card{\operatorname{card}}

\begin{document}
\title[Zeta functions for meromorphic functions]{Zeta functions and Oscillatory Integrals for meromorphic functions}
\author{Willem Veys}
\address{University of Leuven, Department of Mathematics\\
Celestijnenlaan 200 B, B-3001 Leuven, Belgium}
\email{wim.veys@wis.kuleuven.be}
\thanks{The first author was
partially supported by the KU Leuven grant OT/11/069}
\author{W. A. Z\'{u}\~{n}iga-Galindo}
\address{Centro de Investigaci\'{o}n y de Estudios Avanzados, Departamento de Matem\'{a}ticas,
Unidad Quer\'etaro, Libramiento Norponiente \#2000, Fracc. Real de
Juriquilla. C.P. 76230, Quer\'etaro, QRO, M\'exico. }
  \email{wazuniga@math.cinvestav.edu.mx}
\thanks{The second author was partially supported by the ABACUS Project, EDOMEX-2011-C01-165873}

\begin{abstract}
In the 70's Igusa developed a uniform theory for local zeta
functions and oscillatory integrals attached to polynomials with
coefficients in a local field of characteristic zero. In the present
article this theory is extended to the case of rational functions,
or, more generally, meromorphic functions $f/g$, with coefficients
in a local field of characteristic zero. This generalization is far
from being straightforward due to the fact that several new
geometric phenomena appear. Also, the oscillatory integrals have two
different asymptotic expansions: the \lq usual\rq\ one when the norm
of the parameter tends to infinity, and another one when the norm of
the parameter tends to zero. The first asymptotic expansion is
controlled by the poles (with negative real parts) of all the
twisted local zeta functions associated to the meromorphic functions
$f/g - c$, for certain special values $c$. The second expansion is
controlled by the poles (with positive real parts) of all the
twisted local zeta functions associated to $f/g$.

\end{abstract}
\maketitle

\section{introduction}

In the present article we extend the theory of local zeta functions
and oscillatory integrals to the case of {\em meromorphic} functions
defined over local fields. In the classical setting, working on a
local field $K$ of characteristic zero, for instance on
$\mathbb{R}$, $\mathbb{C}$ or $\mathbb{Q}_{p}$, one considers a pair
$(h,\Phi)$, where $h:U\rightarrow K$ is a $K$-analytic function
defined on an open set $U \subset K^n$ and $\Phi$ is a test
function, compactly supported in $U$. One attaches to $(h,\Phi)$ the
local zeta function
\[
Z_{\Phi}(s;h):=\int\limits_{K^{n}\smallsetminus h^{-1}\left\{
0\right\} } \Phi\left( x\right) \left\vert h\left( x\right)
\right\vert _{K} ^{s}\left\vert dx\right\vert_K
\]
for $s\in \mathbb{C}$ with $\operatorname{Re}(s)>0$. Here
$\left\vert dx\right\vert_K$ denotes the Haar measure on $K^n$ and
$\left\vert a\right\vert_K$ the modulus of $a\in K$. More generally,
one considers the (twisted) local zeta function
\[
Z_{\Phi}(\omega;h):=\int\limits_{K^{n}\smallsetminus h^{-1}\left\{
0\right\} } \Phi\left( x\right) \omega \big(h(x)\big) \left\vert
dx\right\vert_K
\]
for a quasi-character $\omega$ of $K^\times$. By using an embedded
resolution of singularities of $h^{-1}\{ 0\}$, one shows that
$Z_{\Phi}(\omega;h)$ admits a meromorphic continuation to the whole
complex plane. In the Archime\-dean case the poles of
$Z_{\Phi}(\omega;h)$ are integer shifts of the roots of the
Bernstein-Sato polynomial of $h$, and hence induce eigenvalues of
the complex monodromy of $h$. In the $p$-adic case such a connection
has been established in several cases, but a general proof
constitutes one of the major challenges in this area.

In the Archimedean case $K=\mathbb{R}$ or $\mathbb{C}$, the study of
local zeta functions was initiated by I. M. Gel'fand and G. E.
Shilov \cite{G-S}. The meromorphic continuation of the local zeta
functions was established, independently, by M. Atiyah \cite{Atiyah}
and J. Bernstein \cite{Ber}, see also \cite[Theorem 5.5.1 and
Corollary 5.5.1]{I2}. On the other hand, in the middle 60's, A. Weil
initiated the study of local zeta functions, in the Archimedean and
non-Archimedean settings, in connection with the Poisson-Siegel
formula \cite{We0}. In the 70's, J.-I. Igusa developed (for
polynomials $h$) a uniform theory for local zeta functions over
local fields of characteristic zero \cite{I3}, \cite{I2}. The
$p$-adic local zeta functions are connected with the number of
solutions of polynomial congruences $\operatorname{mod} p^{m\text{
}}$and with exponential sums $\operatorname{mod}p^{m}$, see e.g.
\cite{D0}, \cite{I2}. More recently, J. Denef and F. Loeser
introduced in \cite{D-L} the motivic zeta functions which constitute
a vast generalization of the $p$-adic local zeta functions.

\bigskip

Fixing an additive character $\Psi:K\rightarrow\mathbb{C}$, the
oscillatory integral attached to $\left( h,\Phi\right) $ is defined
as
\[
E_{\Phi}\left( z;h\right) =\int\limits_{K^{n}\smallsetminus h^{-1}\{
0\}} \Phi\left( x\right) \Psi\big( zh( x)\big) \left\vert
dx\right\vert_K \qquad \text{ for }z\in K.
\]
A central mathematical problem consists in describing the asymptotic
behavior of $E_{\Phi}\left( z;h\right) $ as $\left\vert z\right\vert
_{K} \rightarrow\infty$. Under mild conditions, there exists an
asymptotic expansion of $E_{\Phi}\left( z;h\right) $ at infinity,
controlled by the poles of twisted local zeta functions. For
instance when $K=\mathbb{Q}_p$ we have the following \cite{I3}.
Assume that (the restriction to the support of $\Phi$ of) the
critical locus of $h$ is contained in $h^{-1}\{ 0\}$, and denote for
simplicity $|z|=|z|_{\mathbb{Q}_p}$. Then
\begin{equation}
E_\Phi\left( z;h\right) = \sum_\gamma \sum_{m=1}^{m_{\gamma}}
e_{\gamma,m}\left( \frac{z}{\left\vert z\right\vert }\right)
\left\vert z\right\vert^{\gamma}\left( \ln\left\vert z\right\vert
\right) ^{m-1} \label{intro1}
\end{equation}
 for sufficiently large $|z|$, where
$\gamma$ runs through all the poles$\mod 2\pi i / \ln p$ of
\linebreak $Z_\Phi\left(\omega;h\right)$ (for all quasi-characters
$\omega$), $m_\gamma$ is the order of $\gamma$, and $e_{\gamma,m}$
is a locally constant function on $\mathbb{Z}_p^\times$.

In this article we extend Igusa's theory for local zeta functions
and oscillatory integrals to the case in which the function $h$ is
replaced by a meromorphic function. (Actually, Igusa's theory is
developed in \cite{I3}, \cite{I2} for polynomials $h$, but it is
more generally valid for analytic functions, since the arguments are
locally analytic on a resolution space.) Besides independent
interest, there are other mathematical and physical motivations for
pursuing this line of research. S. Gusein-Zade, I. Luengo and A.
Melle-Hern\'andez have studied the complex monodromy (and A'Campo
zeta functions attached to it) of meromorphic functions, see e.g.
\cite{G-LM}, \cite{G-LM1}, \cite{G-LM2}. This work drives naturally
to ask about the existence of local zeta functions with poles
related with the monodromies studied by the mentioned authors.
 At an arithmetic level, we mention the special case of the oscillatory integrals associated to
$p$-adic Laurent polynomials, that are if fact exponential sums
$\operatorname{mod} p^{\ell}$. Estimates for exponential sums
attached to non-degenerate Laurent polynomials $\operatorname{mod}
p$ were obtained by A. Adolphson and S. Sperber \cite{A-S2} and J.
Denef and F. Loeser \cite{D-L-1}. Finally, the local zeta functions
attached to meromorphic functions are very alike to parametric
Feynman integrals and to $p$-adic string amplitudes, see e.g.
\cite{Be-Bro}, \cite{Bo-We}, \cite{B-F-O-W}, \cite{Marcolli}. For
instance in \cite[Section 3.15]{Marcolli}, M. Marcolli pointed out
explicitly that the motivic Igusa zeta function constructed by J.
Denef and F. Loeser may provide the right tool for a motivic
formulation of the dimensionally regularized parametric Feynman
integrals.

\bigskip
We now describe our results. In Section \ref{section zetafunctions},
we describe the meromorphic continuation of (twisted) local zeta
functions attached to meromorphic functions $f/g$, see Theorems
\ref{Theorem1} and \ref{Theorem2}. These results are related with
the work of F. Loeser in \cite{Lo2} for multivariable local zeta
functions. The local zeta functions attached to meromorphic
functions may have poles with positive and negative real parts. We
establish the existence of a band
$\beta<\operatorname{Re}(s)<\alpha$, with $\alpha \in
\mathbb{Q}_{>0} \cup \{+\infty\}$ and $\beta \in \mathbb{Q}_{<0}
\cup \{-\infty\}$, on which the local zeta functions are analytic,
and we show that $\alpha$, $\beta$ can be read of from an embedded
resolution of singularities of the divisor defined by $f$ and $g$,
see Theorem \ref{Theorem3}. In the case of meromorphic functions,
the problem of determining whether or not the corresponding local
zeta function has a pole is more complicated than in the function
case, see Example \ref{example}. We develop criteria for the
existence of poles, see Corollary \ref{coroll} and Lemma
\ref{Lemma7}.
 In Subsection \ref{motivic}, we treat briefly the
motivic and topological zeta functions for rational functions. These
invariants are connected with special cases of the general theory of
motivic integration of R. Cluckers and F. Loeser \cite{CL1}.

%, see Theorem 3.19.
%We want to mention here that Raibaut [38] introduced a different kind of motivic zeta function for a rational function, and that in dimension 2
% González Villa and Lemahieu studied the monodromy conjecture for the topological zeta function attached to a
%rational function. Our topological zeta functions are way more general than the ones considered by the mentioned authors.

\smallskip

Sections \ref{section oscillatory} and \ref{section expansions}
constitute the core of the article. In Section \ref{section
oscillatory}, we review Igusa's method for estimating oscillatory
integrals for polynomials/holomorphic functions, and we present our
new strategy and several technical results in the case of
meromorphic functions, see Propositions \ref{propE0},
\ref{PropEinfty}, \ref{Prop1p} and \ref{Prop2p}.
 Then in Section \ref{section
expansions}, we prove our expansions and estimations for oscillatory
integrals associated to meromorphic functions. This is not a
straightforward matter, due to the fact that new geometric phenomena
occur in the meromorphic case, see Remarks \ref{zeta2},
\ref{conductor} and \ref{diffform}, and Definitions
\ref{specialpoints} and \ref{specialvalues}. In addition, there
exist two different asymptotic expansions: one when the parameter of
the oscillatory integral approaches the origin and another when this
parameter approaches infinity. The first asymptotic expansion is
controlled by the poles with positive real parts of all twisted
local zeta functions attached to the corresponding meromorphic
function, see Theorems \ref{Esmallz}, \ref{Esmallz-p} and
\ref{estimate0}. The second expansion is controlled by the poles
with negative real parts of all twisted local zeta functions
attached to the corresponding meromorphic function, see Theorems
\ref{thmEinfty}, \ref{thmEinfty-p} and \ref{estimateinfty}. As an
illustration, we mention here the form of the expansion at infinity
when $K=\mathbb{Q}_p$. Let now $h$ be a meromorphic function on $U$.
In Definition \ref{specialvalues} we associate a finite set
$\mathcal {V}$ of {\em special values} to $h$ and $\Phi$, including
the critical values of (a resolution of indeterminacies of) $h$ that
belong to the support of $\Phi$. Then
\begin{equation}
E_{\Phi}\left( z;h\right) = \sum_{c\in \mathcal {V}}
{\displaystyle\sum\limits_{\gamma_c<0}}\
{\displaystyle\sum\limits_{m=1}^{m_{\gamma_c}}}
e_{\gamma_c,m,c}\left( \frac{z}{\left\vert z\right\vert }\right)
\Psi(c\cdot z)
 \left\vert z\right\vert^{\gamma_c}\left( \ln \left\vert z\right\vert\right)^{m-1}
 \label{intro2}
\end{equation}
 for sufficiently large $\left\vert z\right\vert$,
where $\gamma_c$ runs trough all the poles$\mod 2\pi\sqrt{-1}/\ln p$
with negative real part of $Z_{\Phi}\left( \omega;h-c\right) $ (for
all quasi-characters $\omega$), $m_{\gamma_c}$ is the order of
$\gamma_c$, and $e_{\gamma_c,m,c}$ is a locally constant function on
$\mathbb{Z}_p^\times$. Consider for example $h=(x^2+x^3-y^2)/x^2$.
Then $1$ is the only special value (it is not a critical value), and
Theorem \ref{thmEinfty-p} yields the term $\Psi( z) \left\vert
z\right\vert^{-5/2}$ in the expansion, see Examples
\ref{examplecritical} and \ref{examplecritical2}.

Note that the expansion (\ref{intro1}) is simpler, mainly due to
restricting the support of $\Phi$. In the context of meromorphic
functions however, one typically has expansions as in
(\ref{intro2}), even when the support of $\Phi$ is arbitrarily
small!

\smallskip
In \cite{CL}, see also \cite{CGH}, R. Cluckers and F. Loeser
obtained similar expansions to those presented in Theorem
\ref{thmEinfty-p}, for $p$ big enough, in a more general
non-Archimedean setting, but without information on the powers of
$|z|_{K}$ and $\ln|z|_{K}$. We also note that the existence of two
types of asymptotic expansions for $p$-adic oscillatory integrals
attached to Laurent polynomials, satisfying certain very restrictive
conditions, was established by E. Le\'on-Cardenal and the second
author in \cite{L-Z}.

% It is worth to mention that the method for estimating
%oscillatory integrals at infinity is partly based on ideas
%introduced by J. Denef and the first author in \cite{DV},

% As an application of our results, we obtain a general estimation for exponential sums $\operatorname{mod}p^{\ell}$
%attached to Laurent polynomials, see Corollary \ref{laurent}.

 \bigskip

Acknowledgement. We thank Raf Cluckers, Arno Kuijlaars, Adam
Parusinski, Michel Raibaut and Orlando Villamayor for helpful
conversations.

\section{preliminaries}\label{section preliminaries}

\subsection{The group of quasicharacters of a local field}

We take $K$ to be a non-discrete$\ $locally compact field of
characteristic zero. Then $K$ is $\mathbb{R}$, $\mathbb{C}$, or a
finite extension of $\mathbb{Q}_{p}$, the field of $p$-adic numbers.
If $K$ is $\mathbb{R}$ or $\mathbb{C}$, we say that $K$ is an
$\mathbb{R}$\textit{-field}, otherwise we say that $K$ is a
$p$\textit{-field}.

For $a\in K$, we define the \textit{modulus} $\left\vert a\right\vert _{K}$ of
$a$ by%
\[
\left\vert a\right\vert _{K}=\left\{
\begin{array}
[c]{l}%
\text{the rate of change of the Haar measure in }(K,+)\text{ under
}x\rightarrow ax\text{ }\\
\text{for }a\neq0,\\
\\
0\ \text{ for }a=0\text{.}%
\end{array}
\right.
\]

\noindent
 It is well-known that, if $K$ is an $\mathbb{R}$-field, then
$\left\vert a\right\vert _{\mathbb{R}}=\left\vert a\right\vert $ and
$\left\vert a\right\vert _{\mathbb{C}}=\left\vert a\right\vert
^{2}$, where $\left\vert \cdot\right\vert $ denotes the usual
absolute value in $\mathbb{R}$ or $\mathbb{C}$, and, if $K$ is a
$p$-field, then $\left\vert \cdot\right\vert _{K}$ is the normalized
absolute value in $K$.

A quasicharacter of $K^{\times}$ is a continuous group homomorphism
from $K^{\times}$ into $\mathbb{C}^{\times}$. The set of
quasicharacters forms a topological Abelian group denoted by
$\Omega\left( K^{\times}\right) $. The connected component of the
identity consists of the $\omega_{s}$, $s\in\mathbb{C}$, defined by
$\omega_{s}\left( z\right) =\left\vert z\right\vert _{K}^{s}$ for
$z\in K^{\times}$.

\medskip
We now take $K$ to be a $p$-field. Let $R_{K}$\ be the valuation
ring of $K$, $P_{K}$ the maximal ideal of $R_{K}$, and
$\overline{K}=R_{K}/P_{K}$ \ the residue field of $K$. The
cardinality of
the residue field of $K$ is denoted by $q$, thus $\overline{K}=\mathbb{F}_{q}%
$. For $z\in K$, $ord\left( z\right) \in\mathbb{Z}\cup\{+\infty\}$ \
denotes the valuation of $z$, and $\left\vert z\right\vert
_{K}=q^{-ord\left( z\right) }$. We fix a uniformizing parameter
$\mathfrak{p}$ of $R_{K}$. If $z\in K^{\times}$, then $ac$
$z=z\mathfrak{p}^{-ord(z)}$ denotes the angular component of $z$.

Given $\omega\in \Omega\left( K^{\times}\right) $, we choose
$s\in\mathbb{C}$\ satisfying $\omega\left( \mathfrak{p}\right)
=q^{-s}$. Then $\omega$ can be described as follows: $\omega\left(
z\right) =\omega_{s}\left( z\right) \chi\left( ac\text{ }z\right) $
for $z\in K^{\times}$, in which $\chi:=\omega\mid _{R_{K}^{\times}}$
is a character of $R_{K}^{\times}$. Furthermore, $\Omega\left(
K^{\times}\right) $ is a one dimensional complex manifold,
isomorphic to $(\mathbb{C}\text{ mod }(2\pi i/\ln
q)\mathbb{Z})\times\left( R_{K}^{\times}\right) ^{\ast}$, where
$\left( R_{K}^{\times}\right) ^{\ast}$ is the group of characters of
$R_{K}^{\times}$.

Next we take $K$ to be an $\mathbb{R}$-field. Now for $z\in
K^{\times}$ we denote $ac$ $z=\frac{z}{\left\vert z\right\vert }$.
Then $\omega \in \Omega\left( K^{\times}\right) $ can again be
described as $\omega(z)=\omega_{s}\left( z\right) \chi\left(
ac\text{ }z\right) $ for $z\in K^{\times}$, where $\chi$ is a
character of $\{z \in K^{\times} \mid |z|=1\}$. Concretely in this
case
%$\omega\left( z\right) =\left\vert z\right\vert _{K}^{s}\left(\frac{z}{\left\vert z\right\vert }\right) ^{l}$,
$\chi=\chi_l=(\cdot)^l$,
 in which
 $l\in \{0,1\}$ or $l\in\mathbb{Z}$, according
as $K=\mathbb{R}$ or $K=\mathbb{C}$. In addition, $\Omega\left(
K^{\times}\right) $ is a one dimensional complex manifold, which is
isomorphic to
$\mathbb{C}\times\left(  \mathbb{Z}/2\mathbb{Z}\right)  $ or $\mathbb{C}%
\times\mathbb{Z}$, according as $K$ is $\mathbb{R}$ or $\mathbb{C}$.

\medskip
For arbitrary $K$, we will denote the decompositions above as
$\omega =\omega_{s}\chi\left( ac\right) \in\Omega\left(
K^{\times}\right)$. We have that $\sigma\left( \omega\right)
:=\operatorname{Re}(s)$ depends only on $\omega$, and $\left\vert
\omega\left( z\right) \right\vert =\omega_{\sigma\left(
\omega\right) }\left( z\right) $. We define for all $\beta <\alpha$
in $\mathbb{R}\cup \{-\infty,+\infty\}$ an open subset of
$\Omega\left( K^{\times}\right) $ by
\[
\Omega_{\left( \beta,\alpha\right) }\left( K^{\times}\right)
=\left\{ \omega\in\Omega\left( K^{\times}\right) \mid\sigma\left(
\omega\right) \in\left( \beta,\alpha\right) \right\} .
\]
For further details we refer the reader to \cite{I2}.

\subsection{Local zeta functions for meromorphic
functions}\label{notations}

If $K$ is a $p$-field, resp. an $\mathbb{R}$-field, we denote by
$\mathcal{D}(K^{n})$ the $\mathbb{C}$-vector space consisting of all
$\mathbb{C}$-valued locally constant functions, resp. all smooth
functions, on $K^n$, with compact support. An element of
$\mathcal{D}(K^{n})$ is called a \textit{test function}. To simplify
terminology, we will call a non-zero test function that takes only
real and non-negative values a \textit{positive}\ test function.

Let $f,g:U\rightarrow K$ be non-zero $K$-analytic functions defined
on an open $U$ in $K^{n}$, such that $f/g$ is not constant.
Let $\Phi:U\rightarrow \mathbb{C}$ be a test function with support
in $U$. Then the local zeta function attached to $\left(
\omega,f/g,\Phi\right) $ is defined as
\begin{equation}
Z_{\Phi}(\omega;f/g)=\int\limits_{U\smallsetminus D_{K}}\Phi\left(
x\right) \omega\left( \frac{f\left( x\right) }{g\left( x\right)
}\right) |dx|_{K}, \label{zeta}
\end{equation}
where $D_{K}=f^{-1}\left\{ 0\right\} \cup g^{-1}\left\{ 0\right\} $
and $|dx|_{K}$ is the normalized Haar measure on $K^{n}$.

\begin{remark}\rm
(1) The convergence of the integral in (\ref{zeta}) is not a
straightforward matter; in particular the convergence does not
follow from the fact that $\Phi$\ has compact support.

(2) When considering only polynomials $f$ and $g$, it would be
natural to assume that $f$ and $g$ are coprime in the polynomial
ring $K[x_1,\dots,x_n]$. In that case the set $D_K$ only depends on
$f/g$. For more general $K$-analytic functions however, the set
$D_K$ depends in fact on the chosen $f$ and $g$ to represent the
meromorphic function $f/g$. But this will not really affect our
methods and results. Note that the zeta function
$Z_{\Phi}(\omega;f/g)$ does depend only on the quotient $f/g$.
%(3) MAYBE SOMETHING LIKE THIS. BUT I ONLY KNOW REFERENCES OVER R AND C IN GRIFFITHS AND HARRIS.
% Working locally in the sense of germs in a point $P$ of $K^{n}$,  we could ask that
%the germs at $P$ of $f$ and $g$ are coprime in the local ring at $P$. Then (see ....) their germs are still coprime in the local ring
%of points in some small neighborhood of $P$. This would be useful if the support of $\Phi$ is contained in such a neighborhood.
\end{remark}

\subsection{Ordinary and adapted embedded resolutions}

We state two versions of embedded resolution of $D_{K}$
% (or the function $\frac{f}{g}$)
 that we will use in this paper.

\begin{theorem}
\label{thresolsing} %CHECK ! ALSO OK IN p-ADIC CASE?
 Let $U$ be an open subset
%(CHECK)
 of $K^{n}$. Let
$f,g$ be $K$-analytic functions on $U$ as in Subsection
\ref{notations}.

\smallskip
 (1) Then there exists an embedded resolution
$\sigma:X_{K}\rightarrow U$ of $D_{K}$, that is,

\noindent(i) $X_{K}$ is an $n$-dimensional $K$-analytic manifold,
$\sigma$ is a proper $K$-analytic map which is locally a composition
of a finite number of blow-ups at closed submanifolds, and which is
an isomorphism outside of
% the singular locus of
 $\sigma^{-1}(D_{K})$;

\noindent(ii) $\sigma^{-1}\left( D_{K}\right) =\cup_{i\in T}E_{i}$,
where the $E_{i}$\ are closed submanifolds of $X_{K}$ of codimension
one, each equipped with a pair of nonnegative integers $\left(
N_{f,i},N_{g,i}\right) $ and a positive integer $v_i$,
 satisfying the following. At every  point $b$ of $X_{K}$ there
exist local coordinates $\left( y_{1},\ldots,y_{n}\right) $ on
$X_{K}$ around $b$ such that, if $E_{1},\ldots,E_{r}$ are the
$E_{i}$ containing $b$, we have on some open neighborhood $V$ of $b$
that $E_{i}$ is given by $y_{i}=0$ for $i\in\{1,\ldots,r\}$,
\begin{equation}
\sigma^{\ast}\left( dx_{1}\wedge\ldots\wedge dx_{n}\right)
=\eta\left( y\right) \left(
%TCIMACRO{\dprod \limits_{i\in I}}%
%BeginExpansion
\prod_{i=1}^r
%EndExpansion
y_{i}^{v_{i}-1}\right)  dy_{1}\wedge\ldots\wedge dy_{n}, \label{for2}%
\end{equation}
and
\begin{align}
f^{\ast}(y) & :=\left( f\circ\sigma\right) \left( y\right)
=\varepsilon _{f}\left( y\right) \prod_{i=1}^r
y_{i}^{N_{f,i}},\label{for3}\\
g^{\ast}(y) & :=\left( g\circ\sigma\right) \left( y\right)
=\varepsilon _{g}\left( y\right) \prod_{i=1}^r
y_{i}^{N_{g,i}}, \label{for4}%
\end{align}
where $\eta\left( y\right), \varepsilon_{f}\left( y\right),
\varepsilon_{g}\left(y\right)$ belong to
$\mathcal{O}_{X_{K},b}^{\times}$, the group of units of the local
ring of $X_{K}$\ at $b$.

\smallskip
(2) Furthermore, we can construct such an embedded resolution
$\sigma:X_{K}\rightarrow U$ of $D_{K}$ satisfying the following
additional property at every point $b$ of $X_K$:

\noindent (iii) with the notation of (\ref{for3}) and (\ref{for4}),
either $f^{\ast}(y)$ divides $g^{\ast}(y)$ in
$\mathcal{O}_{X_{K},b}$ (what is equivalent to $N_{f,i} \leq
N_{g,i}$ for all $i=1,\ldots,r$), or conversely $g^{\ast}(y)$
divides $f^{\ast}(y)$.
\end{theorem}

\begin{proof} Part (1) is one of the standard formulations of
embedded resolution. It follows from Hironaka's work \cite{H}. See
also \cite[Section 8]{BEV},\cite[Section
5]{E-N-V},\cite{W1},\cite{W2} for more details, and especially
\cite[Theorem 2.2]{D-vdD} for the $p$-field case.
%that `Property D' is valid in the $p$-adic analytic setting, see 5.11 in \cite{E-N-V}.

%For part (2) we construct first a log-principalization
%$\sigma_{0}:X_{K}^{(0)}\rightarrow U$ of the ideal
%$\mathcal{I}=\left( f,g\right) $, also as a composition of a finite
%number of blow-ups in closed submanifolds, and which is an
%isomorphism outside the support of $\mathcal{I}$. Then we construct
%an embedded resolution $\sigma_{1}:X_{K}\rightarrow X_{K}^{(0)}$ of
%the function $\left( f\circ\sigma_{0}\right) \left(
%g\circ\sigma_{0}\right)$ on $X_{K}^{(0)}$, and take $\sigma=\sigma_0
%\circ \sigma_1$.
%
%Indeed, working locally in a point $a$ of $X_{K}^{(0)}$, let
%$\sigma_0^*(\mathcal{I})$ be generated by the function $h$. Then one
%can write $\sigma_0^*(f) = h\tilde f$ and $\sigma_0^*(g) = h\tilde
%g$ for some functions $\tilde f, \tilde g$. Since the ideal
%$(h\tilde f, h\tilde g)$ is generated by $h$, either $\tilde f$ or
%$\tilde g$ must be a unit (in the local ring of $a$). The desired
%divisibility property is thus already present at the level of
%$X_{K}^{(0)}$; we perform further $\sigma_1$ in order to have
%locally the desired monomial presentations.

One  obtains a resolution as in part (2) by first
resolving the indeterminacies of $f/g$, considered as map from $U$
to the projective line, and then further computing an embedded resolution of the union of the exceptional locus and the strict transform of $D_K$.
\end{proof}

\begin{remark}\label{convention}\rm (1) When $f$ and $g$ are polynomials, the map $\sigma$
is a composition of a finite number of blow-ups. But in the more
general analytic setting, it is possible that one needs infinitely
many blow-ups, and hence that $T$ is infinite. Consider for example
the case $U=\{(a,b)\in \mathbb{R}^2 \mid a\neq 0\}$, $f=(y-1)(y -
\sin 1/x)$ and $g=1$. The curve defined by $f$ has infinitely many
isolated singular points, even contained in a bounded part of $U$.

There are even such examples where the multiplicity of the singular
points is not bounded. Consider the analytic function
 $a(x)=\prod_{j=1}^\infty jx \sin \frac
1{jx}$ on $\mathbb{R}_{>0}$, with zeroes at $\frac 1{m\pi}, m \in
\mathbb{Z}_{>0}$. For instance $\frac 1{m!\pi}$ is a zero of
multiplicity $m$. Then the curve defined by $a(x)-a(y)$ on
$U=\mathbb{R}_{>0}^2$ has infinitely many singular points contained
in a bounded part of $U$, where moreover the multiplicities of those
singular points are not bounded.

\smallskip
(2) However, the construction of $\sigma$ is locally finite in the
following sense. For any compact set $\mathcal{C}\subset U$, there
is an open neighborhood $U_\mathcal{C} \supset \mathcal{C}$ such
that the restriction of $\sigma$ to $\sigma^{-1}(U_\mathcal{C})$ is
a composition of a finite number of blow-ups. We refer to \cite{W2}
for more details. In order to handle this situation we note that in
our setting

%\noindent
(i) the objects of study associated to $f$ and $g$, zeta functions
and oscillatory integrals, depend only on the values of $f$ and $g$
on a small neighborhood of the (compact) support of the test
function $\Phi$; and

%\noindent
(ii) most of our invariants and estimates depend on $\Phi$.

\smallskip
\noindent {\sc Convention.} We consider here simply the compact set
$\mathcal{C}:=\supp \Phi$ and we fix an appropriate open $U_\Phi :=
U_\mathcal{C}$ such that the restriction of $\sigma$ to
$\sigma^{-1}(U_\Phi)$ is a composition of a finite number of
blow-ups.

%We will only consider test functions $\Phi$ with support in a fixed
%compact set $\mathcal{C}$, and we will assume that the open $U$ in
%Subsection \ref{notations} is in fact an open $U_\mathcal{C}$ as
%above. Hence we may assume that $T$ is finite in Theorem \ref{thresolsing}.
\end{remark}

\begin{definition} \rm
In the sequel we will call an embedded resolution $\sigma$ as in
part (2) of Theorem \ref{thresolsing} an {\em adapted embedded
resolution} of $D_{K}$.
\end{definition}

\begin{remark}\rm
An adapted embedded resolution yields in general much more
components $E_i, i\in T,$ than an \lq economic\rq\ standard embedded
resolution. But it will be a crucial tool to derive the results of
Sections \ref{section oscillatory} and \ref{section expansions}.
\end{remark}

\begin{remark}\label{units}\rm
In Theorem \ref{thresolsing}, we note that $f^{\ast}/g^{\ast}$,
considered as a map to $K$, is defined in the points $b$ of
$\sigma^{-1}\left( D_{K}\right)$ satisfying $N_{f,i} \geq N_{g,i}$
for $i=1,\dots,r$. In particular there can exist points
$b\in\sigma^{-1}\left( D_{K}\right)$ such that
$\frac{f^{\ast}(y)}{g^{\ast}(y)}$ is a unit in the local ring at $b$
(this happens when $N_{f,i}=N_{g,i}$ for $i=1,\dots,r$). Moreover,
in that case the degree of the first non-constant term in the Taylor
expansion at $b$ can be larger than $1$. See Example
\ref{examplecritical1}.

This will
be an important new feature when studying zeta functions and
oscillatory integrals of meromorphic functions $f/g$, compared to
analytic functions $f$.
\end{remark}

\begin{definition}\label{defalphabeta}\rm
Let $\sigma:X_{K}\to U$ be an embedded resolution of $D_{K}$ as in
Theorem \ref{thresolsing} and Remark \ref{convention}. For $i\in T$,
we denote $N_i= N_{f,i} - N_{g,i}$ and call $(N_i,v_i)$ the datum of
$E_i$.

(1) We put $T_+ = \{i\in T\mid N_i>0 \text{ and } E_i \cap
\sigma^{-1}(\supp \Phi) \neq \emptyset\}$ and $T_- = \{i\in T\mid
N_i<0 \text{ and } E_i \cap \sigma^{-1}(\supp \Phi) \neq
\emptyset\}$, and we define
\[
\alpha=\alpha_\Phi\left( \sigma,D_{K}\right) =\left\{
\begin{array}
[c]{ll}%
\min_{i\in T_{-}}\left\{ \frac{v_{i}}{|N_{i}|}\right\} & \text{if
}T_{-}
\neq\emptyset\\
& \\
+\infty & \text{if }T_{-}=\emptyset,
\end{array}
\right.
\]
and
\[
\beta=\beta_\Phi\left( \sigma,D_{K}\right) =\left\{
\begin{array}
[c]{ll}%
\max_{i\in T_{+}}\left\{ -\frac{v_{i}}{N_{i}}\right\} & \text{if
}T_{+}
\neq\emptyset\\
& \\
-\infty & \text{if }T_{+}=\emptyset.
\end{array}
\right.
\]

(2) Whenever $T$ is finite, in particular when $f$ and $g$ are
polynomials, we can remove the condition $E_i \cap \sigma^{-1}(\supp
\Phi) \neq \emptyset$ in the definition of $T_{+}$ and $T_{-}$, and
then we obtain global invariants $\alpha$ and $\beta$ not depending
on $\Phi$.
\end{definition}

%\subsection{Critical set} ADAPT
%We consider $\frac{f}{g}$ as a function
%\[
%\frac{f}{g}:U\smallsetminus \{g=0\}\rightarrow K.
%\]
%Recall that its critical set is
%\[C_{\frac{f}{g}}=
%\left\{ z\in U\smallsetminus \{g=0\} \mid \nabla\left(
%\frac{f}{g}\right) \left( z\right) =0\right\} =\left\{
%\begin{array}
%[c]{c}%
%z\in U\smallsetminus \{g=0\} \mid g\left( z\right) \nabla\left(
%f\right) \left(
%z\right) \\
%=f\left(  z\right)  \nabla\left(  g\right)  \left(  z\right)
%\end{array}
%\right\} .
%\]
%In classical results on zeta functions and oscillatory integrals for
%holomorphic functions $h$ one sometimes needs the condition $C_h
%\subseteq h^{-1}(0)$. In our context we adapt this condition as
%follows.
%  Let $V$ be the open subset of $X_K$ where
%$\frac{f^*}{g^*}$ is defined as a function to $K$; obviously $V$
%contains $\sigma^{-1}(U\setminus \{g=0\})$, but can also contain
%other points of $\sigma^{-1}(D_K)$, see Remark \ref{units}. We will
%sometimes assume that $f/g$ and $\sigma$ satisfy the hypothesis
%\begin{equation}
%C_{\frac{f^*}{g^*}}\subseteq\left( \frac{f^*}{g^*}\right)
%^{-1}\left( 0\right) . \tag{H1}\label{criticalset}
%\end{equation}
%Under this hypothesis we can write $\frac{f^*}{g^*}$ in a
%neighborhood of any point $b$ of $V$ for which
%$\frac{f^*}{g^*}(b)\neq 0$ as
%\[
%\frac{f^*}{g^*}(y) = \frac{f^*}{g^*}(b)\ (1+y_1),
%\]
%where $y=(y_1,\ldots,y_n)$ are appropriate ($K$-analytic)
%coordinates around $b$.

\section{Convergence, meromorphic continuation and poles of local zeta
functions}\label{section zetafunctions}

\subsection{Zeta functions over $p$-fields}\label{subsection zetap}

Before treating general zeta functions for meromorphic functions, it
is useful to recall the following basic computation.

\begin{lemma}
[Lemma 8.2.1 in \cite{I2}]\label{lemma3} Let $K$ be a $p$-field.
Take $a\in K$, $\omega=\omega_s\chi(ac) \in \Omega\left(
K^{\times}\right) $ and $N\in\mathbb{Z}$.
%\smallsetminus}\left\{ 0\right\} $.
 Take also $n\in \mathbb{Z}_{>0}$ and $e\in\mathbb{Z}_{\geq 0}$. Then
\[
{\displaystyle\int\limits_{(a+\mathfrak{p}^{e}R_{K})\smallsetminus\left\{
0\right\} }} \omega\left( z\right) ^{N}\left\vert z\right\vert
_{K}^{n-1}\left\vert dz\right\vert =\left\{
\begin{array}
[c]{llll}%
\left( 1-q^{-1}\right) \frac{q^{-en-eNs}}{1-q^{-n-Ns}} & & \text{if}
&
\begin{array}
[c]{c}%
a\in\mathfrak{p}^{e}R_{K}\\
\chi^{N}=1
\end{array}
\\
&  &  & \\
q^{-e}\omega\left( a\right) ^{N}\left\vert a\right\vert _{K}^{n-1} &
& \text{if} &
\begin{array}
[c]{c}%
a\notin\mathfrak{p}^{e}R_{K}\\
\chi^{N}\mid_{U^{\prime}}=1
\end{array}
\\
&  &  & \\
0 & & &\text{all other cases,}
\end{array}
\right.
\]
in which $U^{\prime}=1+\mathfrak{p}^{e}a^{-1}R_{K}$. In the first
case, the integral converges on $\operatorname{Re}(s)>-\frac{n}{N}$,
if $N>0$, and on $\operatorname{Re}(s)<\frac{n}{\left\vert
N\right\vert }$, if $N<0$. Note that for $N=0$ we obtain a non-zero
constant.
\end{lemma}

\begin{theorem}
\label{Theorem1} Assume that $K$ is a $p$-field. We consider $f,g$
as in Section \ref{section preliminaries}, and we fix an embedded
resolution $\sigma$ for $D_K$ as in Theorem \ref{thresolsing}, for
which we also use the notation of Definition \ref{defalphabeta}.
Then the following assertions hold:

\smallskip
\noindent(1) $Z_{\Phi}\left( \omega;f/g\right) $ converges for
$\omega\in$ $\Omega_{\left( \beta,\alpha\right) }\left(
K^{\times}\right) $;

\noindent(2) $Z_{\Phi}\left( \omega;f/g\right) $ has a meromorphic
continuation to $\Omega\left( K^{\times}\right) $ as a rational
function of $\omega\left(\mathfrak{p}\right)=q^{-s} $, and its poles
are of the form
\[
s=- \frac{v_{i}}{N_{i}}+\frac{2\pi\sqrt{-1}}{N_{i}\ln q}k, \ k\in
\mathbb{Z},
\]
%for $i\in T$ such that $N_i\neq 0$.
for $i\in T_+\cup T_-$.
 In addition, the order of any
pole is at most $n$.
\end{theorem}

\begin{proof}
These are more or less immediate consequences of the work of Loeser
on multivariable zeta functions \cite{Lo2}. In his setting $f$ and
$g$ are polynomials, but his arguments are also valid for
holomorphic functions. Take $\omega_1,\omega_2\in\Omega\left(
K^{\times}\right)$. Then, by \cite[Th\'eor\`eme 1.1.4]{Lo2}, the
integral
\begin{equation}
Z_{\Phi}(\omega_1,\omega_2;f,g):=\int\limits_{U\smallsetminus
D_{K}}\Phi\left( x\right) \omega_1(f\left( x\right))\, \omega_2
(g\left( x\right))\,
  |dx|_{K}, \label{zeta2}
\end{equation}
obviously converging when $\sigma(\omega_1)$ and $\sigma(\omega_2)$
are positive, has a meromorphic continuation to $\Omega\left(
K^{\times}\right)\times \Omega\left( K^{\times}\right)$ as a
rational function of $\omega_1\left(\mathfrak{p}\right)=q^{-s_1}$
and $\omega_2\left(\mathfrak{p}\right)=q^{-s_2}$, with more
precisely
\[
\prod_{i}(1-q^{v_i+N_{f,i}s_1+N_{g,i}s_2})
\]
as denominator; here $i$ runs over the $i\in T$ such that $E_i \cap
\sigma^{-1}(\supp \Phi) \neq \emptyset$. Hence the real parts of the
poles of $Z_{\Phi}(\omega_1,\omega_2;f,g)$ belong to the union of
the lines $v_i+N_{f,i}\sigma_1+N_{g,i}\sigma_2=0$. Taking
$\omega_1=\omega$ and $\omega_2=\omega^{-1}$ (and hence $s_1=s$ and
$s_2=-s$), this specializes to the stated results about
$Z_{\Phi}\left( \omega;f/g\right) $.
\end{proof}

\begin{remark}\label{altern}\rm
Alternatively, it is straightforward to adapt Igusa's proof in the
classical case ($g=1$) directly to our situation. As a preparation
for Sections \ref{section oscillatory} and \ref{section expansions},
we recall briefly the main idea.
%, highlighting an important new feature.

(1) We pull back the integral $Z_{\Phi}\left( \omega;f/g\right) $ to
an integral over $X_K\smallsetminus \sigma^{-1}(D_{K})$ via the
resolution $\sigma$, and compute it by subdividing the (compact)
integration domain $\sigma^{-1}\left( \text{supp}\, \Phi \right)
\subset X_{K}$ in a finite disjoint union of compact open sets on
which the integrand becomes \lq monomial\rq\ in local coordinates.
More precisely, using the notation of Theorem \ref{thresolsing}, we
can assume that such an integration domain around a point $b\in X_K$
is of the form $B=c+(\mathfrak{p}^{e}R_{K})^{n}$ in the local
coordinates $y_1,\ldots,y_n$, and that $\vert\eta\left(
y\right)\vert_K$, $\vert\varepsilon_{f}\left( y\right)\vert_K$,
$\vert\varepsilon_{g}\left(y\right)\vert_K$ are constant on $B$.
Then the contribution of $B$ to $Z_{\Phi}\left( \omega;f/g\right) $
is a non-zero constant times
\begin{equation}\label{localcontrib}
%\int_B \omega(y) \left\vert y\right\vert^v\left\vert dy\right\vert
%_{K} =
 \prod_{i=1}^r \int_{c_i+\mathfrak{p}^{e}R_{K}\setminus \{y_i=0\}}
\omega^{N_{i}}\left(y_{i}\right) \left\vert
y_{i}\right\vert_K^{v_{i} -1}\left\vert dy_i\right\vert _{K} .
\end{equation}
%(which is just a constant when $b\not\in \sigma^{-1}(D_K)$).
 Finally one concludes by using the local computation of Lemma \ref{lemma3}.

We want to stress the new feature mentioned in Remark \ref{units}:
in (\ref{localcontrib}), it is possible that all $N_i=0$, while some
$v_i>1$.

(2) Note that one needs an argument as in (1) to see that the
defining integral of $Z_{\Phi}\left( \omega;f/g\right) $ converges
at least somewhere.

(3) If $\sigma$ is an adapted embedded resolution, we have
% above for such a set $B$
around $b\in\sigma^{-1}(D_K)$ that $N_1,\dots,N_r$ are either all
non-positive or all non-negative in (\ref{localcontrib}).
% (when $r>0$).
\end{remark}

\begin{remark}\label{remarkpoles-p} \rm There is a refinement concerning the list of candidate
poles in Theo\-rem \ref{Theorem1} and Remark \ref{altern}. Writing
$\omega= \omega_s\chi(ac)$, we have, by Lemma \ref{lemma3}, that $-
\frac{v_{i}}{N_{i}}$ can be the real part of a pole of the
corresponding integral in (\ref{localcontrib}) only if the order of
$\chi$ divides $N_i$. Hence, in Theorem \ref{Theorem1}(2) the poles
are subject to the additional restriction that the order of $\chi$
divides $N_i$.

For later use in Sections \ref{section oscillatory} and \ref{section
expansions}, we stress the following special case. When $N_i=1$ in
(\ref{localcontrib}), the corresponding integral has no pole unless
$\chi$ is trivial. Then, considering in Theorem \ref{Theorem1} the
case that $g=1$ and $f^{-1}\{0\} \cap \supp \Phi$ has no singular
points, we have that $Z_{\Phi}\left( \omega_s\chi(ac);f\right) $ has
no poles unless $\chi$ is trivial, in which case its poles are of
the form $s=- 1+\frac{2\pi\sqrt{-1}}{\ln q}k, \ k\in \mathbb{Z},$
and of order $1$.
\end{remark}

\subsection{Zeta functions over $\mathbb{R}$-fields}
The strategy being analogous as for $p$-fields, we provide less
details.

\begin{theorem}
\label{Theorem2}Assume that $K$ is an $\mathbb{R}$-field. We
consider $f,g$ as in Section \ref{section preliminaries}, and we fix
an embedded resolution $\sigma$ for $D_K$ as in Theorem
\ref{thresolsing}, for which we also use the notation of Definition
\ref{defalphabeta}. Then the following assertions hold:

\smallskip
\noindent(1) $Z_{\Phi}\left( \omega;f/g\right) $ converges for
$\omega\in$ $\Omega_{\left( \beta,\alpha\right) }\left(
K^{\times}\right) $;

\noindent(2) $Z_{\Phi}\left( \omega,f/g\right) $ has a meromorphic
continuation to $\Omega\left( K^{\times}\right) $, and its poles are
of the form
\[
s= -\frac{v_{i}}{N_{i}}-\frac{k}{[K:\mathbb{R}]N_{i}}, \quad k\in
\mathbb{Z}_{\geq 0},
\]
for $i\in T_+\cup T_-$. In addition, the order of any pole is at
most $n$.
%\noindent(3) If $f/g$ and $\sigma$ satisfy Hypothesis
%\ref{criticalset}, then for every polynomial $P\left( s,m\right) $
%and every vertical strip $B_{\sigma_{1},\sigma_{2}}:=\left\{ s\in
%\mathbb{C} \mid \operatorname{Re}(s)\in\left(
%\sigma_{1},\sigma_{2}\right) \right\} $ we have that
%\[
%P\left( s,m\right) Z_{\Phi}\left( \omega;f/g\right) \text{ is
%uniformly bounded for }s\in\text{ }B_{\sigma_{1},\sigma_{2}},
%\]
%with neighborhoods of the poles of $Z_{\Phi}\left( \omega;f/g\right)
%$ removed therefrom, and for $m\in\mathbb{Z}$.
\end{theorem}

\begin{proof} Analogously as in the proof of Theorem \ref{Theorem1}, this can be derived from the
results on multivariable zeta functions \cite{Lo2}, although there
the detailed form of the possible poles is not mentioned explicitly.
Anyway, (2) can be shown exactly as in the proof of \cite[Theorem
5.4.1]{I2}.
\end{proof}

\begin{remark}\label{alternreal}\rm
Again, as a preparation for Sections \ref{section oscillatory} and
\ref{section expansions}, we mention the main idea of (the
generalization of) Igusa's proof in the classical case.

After pulling back the integral $Z_{\Phi}\left( \omega;f/g\right) $
to an integral over $X_K\smallsetminus \sigma^{-1}(D_{K})$ via the
resolution $\sigma$, subdividing the new integration domain, and
this time also using a partition of the unity,
%$\sigma^{-1}\left(\text{supp}\left( \Phi\right) \right) \subset X_{K}$
one writes $Z_{\Phi}\left( \omega;f/g\right) $ as a finite linear
combination of \lq monomial\rq\ integrals. With the notation of
Theorem \ref{thresolsing}, we can assume that these are of the form
\[
\int_{K^n\setminus \cup_{i=1}^r \{y_i=0\}} \Theta(y)\,
\omega\left(\epsilon_f(y)\epsilon_g^{-1}(y)\right)\prod_{i=1}^r
\omega^{N_{i}}\left(y_{i}\right) \prod_{i=1}^r\left\vert
y_{i}\right\vert_K^{v_{i} -1} \left\vert dy\right\vert _{K},
\]
where $\Theta\left( y\right) $ is a smooth function with support in
the polydisc
\[
\left\{ y\in K^{n}\mid \left\vert y_{j}\right\vert_K <1\text{ for }
j=1,\ldots ,n\right\} .
\]
Note again that, following Remark \ref{units}, it is possible that
all $N_i=0$, while some $v_i>1$.
%The claim then follows by using the
%argument given by Igusa for the analogous classical result, see
%Section 4.2 in \cite{I3}.
\end{remark}

\begin{remark}\label{polesoverC}\rm Looking more in detail at the proof of \cite[Theorem 5.4.1]{I2},
which uses Bernstein polynomial techniques for dealing with the
poles of monomial integrals, we have in fact a somewhat sharper
result concerning the poles of $Z_{\Phi}(\omega;f/g)$ when
$K=\mathbb{C}$: when $\omega\left( z\right) =\left\vert z\right\vert
_{K}^{s}\left( \frac{z}{\left\vert z\right\vert }\right) ^{l}$, they
are of the form $ s= - \frac{|l|}2 -\frac{v_{i}+k}{N_{i}}, k \in
\mathbb{Z}_{\geq 0},$ for $N_i>0$, and $ s= \frac{|l|}2
+\frac{v_{i}+k}{|N_{i}|}, k \in \mathbb{Z}_{\geq 0},$ for $N_i<0$.
\end{remark}

\begin{remark}\rm
For later use in Sections \ref{section oscillatory} and \ref{section
expansions}, we consider in Theorem \ref{Theorem2} the case that
$g=1$ and $f^{-1}\{0\} \cap \supp \Phi$ has no singular points. If
$K=\mathbb{C}$, we have as special case of Remark \ref{polesoverC}
that the poles of $Z_{\Phi}\left( \omega_s\chi_l(ac);f\right) $ are
of the form $s=-\frac {|l|}2 - 1 - k,\ k\in \mathbb{Z}_{\geq 0}$. If
$K=\mathbb{R}$, the poles of $Z_{\Phi}\left(
\omega_s\chi_l(ac);f\right) $ are odd integers when $l=0$ and even
integers when $l=1$. Both when $K=\mathbb{C}$ and $K=\mathbb{R}$
these poles are of order $1$ \cite[Theorem 4.19]{Ja},\cite[II
\S7]{AVG}.
\end{remark}

Here we also want to mention that there is substantial work of
Barlet and his collaborators related to Theorem \ref{Theorem2} and
the remarks above, see e.g. \cite{BM}.

\subsection{Existence of poles, largest and smallest poles}\label{subsection:about poles}

% If $K$ is an $\mathbb{R}$-field, we assume that the topological closure of $U$ is
% compact, and if $K$ is a $p$-field, we take $U$ itself to be compact.
Take $f,g:U\to K$ and a non-zero test function $\Phi$ as in Section
\ref{section preliminaries}. We consider
\[
Z_{\Phi}\left( s;f/g\right) :=Z_{\Phi}\left( \omega_s;f/g\right)=
{\displaystyle\int\limits_{U\smallsetminus D_{K}}}
\Phi\left( x\right) \left\vert \frac{f\left( x\right) }{g\left(
x\right) }\right\vert _{K}^{s}\left\vert dx\right\vert_K
\]
for $s\in \mathbb{C}$. By the results of the previous subsections,
we know that this integral converges when $\beta
<\operatorname{Re}(s)<\alpha$, where $\alpha$ and $\beta$ are as in
Definition \ref{defalphabeta}, hence a priori depending on some
chosen embedded resolution.

It turns out that $\alpha$ and $\beta$ in fact {\em do not} depend
on the chosen resolution. This follows from the next result, that
generalizes the classical result for the zeta function of an
analytic function $f$. For that classical result Igusa gives in
\cite{I1} the strategy of a proof using Langlands' description of
residues in terms of principal value integrals \cite{Lan}, see also
\cite{AVG} for $K=\mathbb{R}$. This idea uses explicit meromorphic
continuations of certain monomial integrals, and it could probably
be extended to the case of meromorphic functions $f/g$. Here we
provide a direct proof that also works simultaneously for all fields
$K$.

\begin{theorem}
\label{Theorem3} Take $K$, $f$, $g$, $\sigma:X_{K}\rightarrow U$ and
$\alpha_\Phi\left( \sigma ,D_{K}\right), \beta_\Phi\left( \sigma
,D_{K}\right)$ as in Section \ref{section preliminaries}.
%be any embedded resolution of singularities of $D_{K}=f^{-1}\left( 0\right) \cup g^{-1}\left(
%0\right) $.

(1) Assume that $\beta_\Phi\left( \sigma ,D_{K}\right) \neq-\infty$,
and that it is equal to $-\frac{v_i}{N_i}$ precisely for $i\in
T_\beta \, (\subset T_+)$. If $\Phi$ is positive with $\Phi(P)>0$
for some $P\in \sigma (\cup_{i\in T_\beta}E_i)$, then
$\beta_\Phi\left( \sigma,D_{K}\right) $ is a pole of $Z_{\Phi}\left(
s;f/g\right) $.

(2) Assume that $\alpha_\Phi\left( \sigma ,D_{K}\right) \neq
+\infty$, and that it is equal to $\frac{v_i}{|N_i|}$ precisely for
$i\in T_\alpha \, (\subset T_-)$. If $\Phi$ is positive with
$\Phi(P)>0$ for some $P\in \sigma (\cup_{i\in T_\alpha}E_i)$, then
$\alpha_\Phi\left( \sigma,D_{K}\right) $ is a pole of
$Z_{\Phi}\left( s;f/g\right) $.

(3) Hence $\alpha_\Phi\left( \sigma,D_{K}\right) $ and
$\beta_\Phi\left( \sigma,D_{K}\right) $ do not depend on $\left(
\sigma,D_{K}\right) $.
\end{theorem}

\begin{proof}
(1) Take a \textit{generic} point $b$ in a component $E$ of
$\sigma^{-1}\left( D_{K}\right) $ with numerical datum $\left(N=
N_{f}-N_{g},v\right) $ such that $\beta=\beta_\Phi\left( \sigma
,D_{K}\right)=-\frac{v}{N}$ and $\Phi(\sigma(b))>0$. Take also a
small enough polydisc-chart $B\subset X_{K}$ around $b$ with
coordinates $y=\left( y_{1},\ldots ,y_{n}\right) $ such that
$\Phi(\sigma(P))>0$ for all $P$ in the closure $\bar B$ of $B$,
\begin{equation}\label{simple monomial}
\left( \frac{f}{g}\circ\sigma \right) \left( y\right)
=\varepsilon\left( y\right) y_{1}^{N} \qquad \text{and}\qquad
\sigma^{\ast}\left( dx_{1} \wedge\ldots\wedge dx_{n}\right)
=\eta\left( y\right) y_{1}^{v-1},
\end{equation}
with $\varepsilon\left( y\right)$ and $\eta\left( y\right)$
invertible power series on $B$. We can change the coordinate $y_1$
in order to have that $\varepsilon$ is a constant, and assume that
$\left\vert \eta\left( y\right) \right\vert _{K}$ is bounded below
by a positive constant $C_1$. When $K$ is a $p$-field, we could
assume that $\left\vert \eta\left( y\right) \right\vert _{K}$ is
constant on $B$. Denoting $C_2= |\varepsilon|_K \min _{Q\in
\sigma(\bar B)}\Phi\left(Q\right) $, we have on $B$ that
$$
( \Phi \circ \sigma)(y) \left\vert \left(
\frac{f}{g}\circ\sigma\right) \left( y\right) \right\vert_{K}^\beta
\left\vert \eta\left( y\right) \right\vert _{K} |y_{1}^{v-1}|_K \geq
C_2\, |y_1|_K^{N\beta}\, C_1 |y_1|_K^{v-1} = C_2 C_1 |y_1|_K^{-1}.
$$
Then, by the $K$-analytic change of variables formula (see e.g.
\cite[Proposition 7.4.1]{I2} for $p$-fields),
\begin{align*}
{\displaystyle\int\limits_{U\smallsetminus D_{K}}} \Phi\left(
x\right) \left\vert \frac{f\left( x\right) }{g\left( x\right)
}\right\vert _{K}^{\beta}\left\vert dx\right\vert_K
 & =
{\displaystyle\int\limits_{X_K\smallsetminus \sigma^{-1}D_{K}}}
(\Phi\circ\sigma) \left( y\right) \left\vert \left(
\frac{f}{g}\circ\sigma\right) \left( y\right) \right\vert_{K}^\beta
\left\vert \sigma^*dx\right\vert_K\\
&\geq C_2 C_1 {\displaystyle\int\limits_{B\setminus \{y_1=0\}}}
|y_1|_K^{-1}|dy|_K = +\infty.
\end{align*}
Note that $Z_{\Phi}(\gamma;f/g)$ exists and is positive for $\gamma
\in \mathbb{R}$ satisfying $\beta<\gamma<0$. Then, by using the Monotone
Convergence Theorem, we have that
$$
\lim_{\gamma\to\beta} Z_{\Phi}(\gamma;f/g) = \lim_{\gamma\to\beta}
{\displaystyle\int\limits_{U\smallsetminus D_{K}}} \Phi\left(
x\right) \left\vert \frac{f\left( x\right) }{g\left( x\right)
}\right\vert _{K}^{\gamma}\left\vert dx\right\vert_K =+\infty.
$$
Hence $\beta$ must be a pole of $Z_{\Phi}(s;f/g)$, being a meromorphic function in the whole complex plane.

%CHECK IF CAN FORMULATE END SHORTER/BETTER.

(2) By Part (1) we have that $-\alpha_\Phi\left( \sigma,D_{K}\right)
$ is
 a pole of $Z_{\Phi}\left( s;g/f\right) $, and
hence $\alpha_\Phi\left( \sigma,D_{K}\right) $ is a pole of
 $Z_{\Phi}\left( s;f/g\right) =Z_{\Phi}\left(-s;g/f\right) $.

\end{proof}

\begin{remark}\rm
Note that, whenever $T$ is finite, in particular when $f$ and $g$
are polynomials, the global $\alpha$ and $\beta$ of Definition
\ref{defalphabeta}(2) also do not depend on $\Phi$.
\end{remark}

\begin{remark}\rm
In \cite[Theorem 2.7]{VZ} we showed in the context of zeta functions
of mappings over $p$-fields that the analogue of $\beta$ is a real
part of a pole. With the technique above one can in fact show that
also in that context $\beta$ itself is a pole.
\end{remark}

%In this section we work with $\mathbb{R}$-fields exclusively. Given
%$h:W\rightarrow K$ a $K$-analytic function on $W$, an open subset of
%$K^{n}$,
%we set%
%\[
%V_{h}\left(  K;W\right)  :=\left\{  x\in W;h\left(  x\right)  =0\right\}
%\]
%and%
%\[ V_{h}^{reg}\left( K;W\right) :=\left\{ x\in W;h\left( x\right)
%=0,\nabla h\left( x\right) \neq0\right\} .
%\]

\begin{corollary}
\label{coroll}Take $\frac{f}{g}:U\smallsetminus D_{K}\rightarrow K$
as in Section \ref{section preliminaries}.

(1) Assume that there exists a point $x_{0}\in U$ such that $f\left(
x_{0}\right) =0$ and $g\left( x_{0}\right) \neq0$. Then, for any
positive test function $\Phi$ with support in a small enough
neighborhood of $x_0$, the zeta function $Z_{\Phi}\left(
s;f/g\right) $ has a negative pole.
% (resp. a pole with negative real part) when $K$ is an $\mathbb{R}$-field (resp. a $p$-field).

(2) Assume that there exists a point $x_{0}\in U$ such that $f\left(
x_{0}\right) \neq 0$ and $g\left( x_{0}\right) =0$. Then, for any
positive test function $\Phi$ with support in a small enough
neighborhood of $x_0$, the zeta function $Z_{\Phi}\left(
s;f/g\right) $ has a positive pole.
% (resp. a pole with positive real part) when $K$ is an $\mathbb{R}$-field (resp. a $p$-field).

(3) In particular, if $K=\mathbb{C}$ and $f$ and $g$ are
polynomials, then $Z_{\Phi}\left( s;f/g\right) $ always has a pole
for an appropriate positive test function $\Phi$.
\end{corollary}

\begin{proof}
Given $x_0$ as in (1), there must exist a component $E_i$ of
$\sigma^{-1}(D_K)$ in an embedded resolution $\sigma$ of $D_K$ for
which $N_i>0$. Hence $\beta\left( \sigma,D_{K}\right) \neq - \infty$
and then, by Theorem \ref{Theorem3}, it is a pole. Alternatively,
$f/g$ is analytic in a neighborhood of $x_0$, and hence this is
really just the classical result.

The argument for (2) is similar.
\end{proof}

\smallskip
When $f$ and $g$ have exactly the same zeroes in $U$, we cannot
derive any conclusion from Corollary \ref{coroll}. In fact, in that
case all possibilities can occur: no poles, only positive poles,
only negative poles, or both positive and negative poles. We provide
an example for each situation.

\begin{example}\label{example}\rm
Let $K=\mathbb{R}$, or $K=\mathbb{Q}_p$ such that $-1$ is not a
square in $\mathbb{Q}_p$ (that is, $p \equiv 3 \mod 4$). In the
examples below we take $U=K^2$ and we consider in each case
polynomials $f$ and $g$ such that the zero locus over $K$ of both
$f$ and $g$ consists of the origin.

(1) Take $f=(x^{2}+y^{2})^2$ and $g =x^{4}+y^{4}$. An embedded
resolution of $D_K$ is obtained by blowing up at the origin; the
exceptional curve has datum $(N,v)=(0,2)$. Hence, for any test
function $\Phi$, $Z_{\Phi}(s;f/g)$ does not have poles.

(2) Take $f=x^{2}+y^{2}$ and $g =x^{4}+y^{4}$. Again an embedded
resolution of $D_K$ is obtained by blowing up at the origin; the
exceptional curve now has datum $(N,v)=(-2,2)$. Hence $1$ is a pole
of $Z_{\Phi}(s;f/g)$ when $\Phi$ is positive around the origin, and
there are no negative (real parts of) poles.

(3) Take $f=x^{4}+y^{4}$ and $g=x^{2}+y^{2}$. Analogously, $-1$ is a
 pole of $Z_{\Phi}(s;f/g)$ when $\Phi$ is positive
around the origin, and there are no positive (real parts of) poles.

(4) Take $f=y^{2}+x^{4}$ and $g =x^{2}+y^{4}$. Now an embedded
resolution of $D_K$ is obtained by first blowing up at the origin,
yielding an exceptional curve with datum $(0,2)$, and next
performing two blow-ups at centres the origins of the two charts,
yielding exceptional curves with data $(2,3)$ and $(-2,3)$. Hence
$-3/2$ and $3/2$ are poles of $Z_{\Phi}(s;f/g)$ when $\Phi$ is
positive around the origin.
\end{example}

%\subsection{Criteria for the existence of poles}\label{subsection:criteria poles}

\begin{lemma}
\label{Lemma7} Take $\frac{f}{g}:U\smallsetminus D_{K}\rightarrow K$
as in Section \ref{section preliminaries}.
% Let $\Phi$ be a (non-trivial) positive test function.

(1) The zeta function $Z_{\Phi}\left( s;f/g\right) $ has a negative
pole,
% (resp. a pole with negative real part),
 for an appropriate
positive $\Phi$, if and only if there exists a bounded sequence
$\left( a_{n}\right)_{n\in\mathbb{Z}_{\geq 0}} \subset
U\smallsetminus D_{K}$ such that
\begin{equation}
\lim_{n\rightarrow\infty}\left\vert \frac{f\left( a_{n}\right)
}{g\left( a_{n}\right) }\right\vert _{K}=0. \label{limit_1}
\end{equation}

(2) The zeta function $Z_{\Phi}\left( s;f/g\right) $ has a positive
pole,
% (resp. a pole with positive real part),
for an appropriate positive $\Phi$, if and only if there exists a
bounded sequence $\left( a_{n}\right)_{n\in\mathbb{Z}_{\geq 0} }
\subset U\smallsetminus D_{K}$ such that
\begin{equation}
\lim_{n\rightarrow\infty}\left\vert \frac{f\left( a_{n}\right)
}{g\left( a_{n}\right) }\right\vert _{K}=\infty. \label{limit_1}
\end{equation}

\end{lemma}

\begin{proof} (1)
We consider an adapted embedded resolution $\sigma$ of $D_K$ as in
Theorem \ref{thresolsing}, obtained by resolving the indeterminacies
of $f/g$. From Theorem \ref{Theorem3}, it is clear that
$Z_{\Phi}\left( s;f/g\right) $ has a negative pole
% (resp. a pole with negative real part)
for an appropriate positive $\Phi$ if and
only if there exists a component $E_i$ of $\sigma^{-1}\left(
D_{K}\right) $ with $N_i>0$.

Assume first that $N_i>0$ for the component $E_i$. Take a generic
point $b$ of $E_i$ and a small enough chart $B\subset X_{K}$ around
$b$ with coordinates $y=\left( y_{1},\ldots ,y_{n}\right) $, such
that the equation of $E_i$ around $b$ is $y_1=0$. Take a sequence
$(a'_n)_{n\in \mathbb{Z}_{\geq 0}}$ defined by
$a'_n=(z_n,0,\dots,0)$, where all $z_n\neq 0$ and $
\lim_{n\rightarrow\infty}\left\vert z_n\right\vert _{K}=0 $. Clearly
$a'_n$ converges (in norm) to $b$, and hence $\frac fg
(\sigma(a'_n))=(\frac fg \circ\sigma) (a'_n)$ converges to $(\frac
fg \circ \sigma) (b)$, and this equals zero since $N_i>0$.
Consequently the sequence $(\sigma(a'_n))_{n\in\mathbb{Z}_{\geq 0}}$
is as desired.

Conversely, assume that the bounded sequence $\left(
a_{n}\right)_{n\in\mathbb{Z}_{\geq 0}} \subset U\smallsetminus
D_{K}$ satisfies $\lim_{n\rightarrow\infty}\left\vert \frac{f\left(
a_{n}\right) }{g\left( a_{n}\right) }\right\vert _{K}=0$. Say the
$a_n$ belong to the compact subset $A$ of $U$. Since $\sigma$ is an
isomorphism outside the inverse image of $D_K$, each $a_n$ has a
unique preimage $a'_n$ in $X_K$. The $a'_n$ belong to the compact
set $\sigma^{-1}(A)$ (recall that $\sigma$ is proper), and hence we
may assume, maybe after restricting to a subsequence, that the
sequence $(a'_n)$ converges (in norm) to an element $a'$ of
$\sigma^{-1}(A)$. The map $\frac fg \circ \sigma$ is everywhere
defined (as map to the projective line) and continuous, yielding
$$
\big(\frac fg \circ\sigma\big)(a')= \big(\frac fg
\circ\sigma\big)(\lim_{n\rightarrow\infty}a'_n)=\lim_{n\rightarrow\infty}\big(\frac
fg \circ\sigma\big)(a'_n)=\lim_{n\rightarrow\infty}\frac fg(a_n)=0.
$$
Consequently $a'$ belongs to the union of the $E_i$ satisfying
$N_i>0$, and then in particular there is at least one such $E_i$.

(2) This is proved analogously.
\end{proof}

\subsection{Denef's explicit formula}

Here we assume that $f$ and $g$ are polynomials over a $p$-field
$K$, and that $\Phi$ is the characteristic function of a union $W$
of cosets$\mod P_K^n$ in $R_K^n$. Denef introduced in \cite{D2} the
notion of {\em good reduction$\mod P_K$} for an embedded resolution,
and proved an explicit formula \cite[Theorem 3.1]{D2} for
$Z_{\Phi}\left(s;f\right)$ when $\sigma$ has good reduction$\mod
P_K$. His arguments extend to the
 case of a rational function $f/g$.

Before stating the formula, we mention the following important
property. When $f$ and $g$ are defined over a number field $F$, one
can choose also the embedded resolution $\sigma$ of $f^{-1}\{0\}\cup
g^{-1}\{0\}$ to be defined over $F$. We can then consider $f/g$ and
$\sigma$ over any non-archimedean completion $K$ of $F$, and then
$\sigma$ has good reduction$\mod P_K$ for all but a finite number of
completions $K$.

%, by replacing resolution by log-principalization (as in Theorem \ref{ThHironaka}).

\begin{theorem}
\label{TheoremDenef} Let $f,g \in R_K\left[x_{1},\ldots
,x_{n}\right]$ such that $f/g$ is not constant, and suppose that
$\Phi$ is the characteristic function of $W$ as above. If $\sigma$
has good reduction$\mod P_K$, then, with the notation of Theorem
\ref{thresolsing},
%as in Theorem \ref {ThHironaka}.
%Denote div$\left( \sigma ^{\ast }\left( \mathcal{I}_{%
%\boldsymbol{f}}\right) \right) =\tsum\nolimits_{i\in T}N_{i}E_{i}$, and div$%
%\left( \sigma ^{\ast }\left( dx_{1}\wedge \ldots \wedge
%dx_{n}\right) \right) =\tsum\nolimits_{i\in T}\left( v_{i}-1\right)
%E_{i},$ where $E_{i}$, $i\in T$, are the irreducible components of
%the simple normal crossings divisor given by the principal ideal $\sigma ^{\ast }\left( \mathcal{I}_{\boldsymbol{f}}\right) $.
%Denote the valuation ring and the residue field of $K$ by $R$ and $\overline{K}$ $=\mathbb{F}_{q}$ respectively.
 we have
\begin{equation*}
Z_{\Phi}\left( s;f/g\right) =q^{-n}\sum\limits_{I\subset
T}c_{I}\prod\limits_{i\in I}\frac{q-1}{q^{v_i +N_{i}s}-1},
\end{equation*}
where
\begin{equation*}
c_{I}=\card\left\{ a\in \overline{X}\left( \overline{K}\right) \mid
a\in \overline{E_{i}}\left( \overline{K}\right) \Leftrightarrow i\in I\text{%
; and }\overline{\sigma }(a)\in \overline{W}\right\} .
\end{equation*}
Here $\overline{\cdot }$ denotes the reduction mod $P_K$, for which
we refer to \cite[Section 2]{D2}.
\end{theorem}

\begin{example}\rm  We compute $Z(s):= \int_{(p\mathbb{Z}_p)^2}
\left|\frac{f(x,y)}{g(x,y)}\right|_{\mathbb{Q}_p}^s |dx\wedge
dy|_{\mathbb{Q}_p}$ for the rational functions $f/g$ in Example
\ref{example}, using Theorem \ref{TheoremDenef}. Recall that
$p\equiv 3\mod 4$.

In case (1) we have simply that $Z(s)= \frac1{p^2}$, but this was
already obvious from the defining integral, since we are just
integrating the constant function $1$. The cases (2) and (3) yield
$$Z(s)=\frac{p^2-1}{p^2(p^{2-2s}-1)} \qquad\text{and}\qquad
Z(s)=\frac{p^2-1}{p^2(p^{2+2s}-1)},$$
 respectively.  For case (4) we obtain
\begin{align*}
Z(s)&=\frac{1}{p^2}\left[ p\frac{p-1}{p^{3+2s}-1} +
p\frac{p-1}{p^{3-2s}-1} + (p-1)\frac{p-1}{p^{2}-1} \right. \\
\qquad &\left. +\frac{(p-1)^2}{(p^{3+2s}-1)(p^2-1)} +
\frac{(p-1)^2}{(p^{3-2s}-1)(p^2-1)}\right] \\
 &=
\frac{(p-1)(p^{4+2s}+p^{4-2s}+p^5-p^4+p^3-p^2-p-1)}{p^2(p^{3+2s}-1)(p^{3-2s}-1)}.
\end{align*}
In addition, we provide an expression for case (4) when $p\equiv
1\mod 4$. Then the strict transform of $f^{-1}\{0\}$ and
$g^{-1}\{0\}$ (with data $(1,1)$ and $(-1,1)$, respectively),
intersects the second and third exceptional curve, respectively, in
two points. The formula of Theorem \ref{TheoremDenef} now yields
nine terms, and we obtain after an elementary computation that
$$
Z(s)=\frac{(p-1)D(p^s,p)}{p^2(p^{3+2s}-1)(p^{3-2s}-1)(p^{1+s}-1)(p^{1-s}-1)}
$$
with
\begin{align*}
D(p^s,p)&= p^{5+3s}+p^{5-3s} + (p^2-2p-1)(p^{4+2s}+p^{4-2s}) \\
+(-p^5&+p^3+p^2-p-1)(p^{1+s}+p^{1-s}) + p^7-p^6+2p^5-2p^4+2p^2+3p-1
.
\end{align*}
\end{example}

\subsection{\label{motivic} Motivic and topological zeta functions}

The analogue of the original explicit formula of Denef plays an
important role in the study of the motivic zeta function associated
to a regular function \cite{D-L}. One can associate more generally a
motivic zeta function to a rational function on a smooth variety,
and obtain a similar formula for it in terms of an embedded
resolution.

Since this is not the focus of the present paper, we just formulate
the more general definition and formula, referring to e.g.
\cite{D-L2}, \cite{V1}\ for the notion of jets and arcs,
Grothendieck ring, and motivic measure. Set as usual $\left[ \cdot
\right] $ for the class of a variety in the Grothendieck ring of
algebraic varieties over a field $F$, and $\mathbb{L}:=\left[
\mathbb{A}^{1}\right] $. We also denote by $\mathcal{M}$ the
localization of that Grothendieck ring with respect to $\mathbb{L}$
and the elements $\mathbb{L}^a -1, a\in \mathbb{Z}_{>0}$.

\smallskip
Let $Y$ be a smooth algebraic variety of dimension $n$ over a field
$F$ of characteristic zero, and $f/g$ a non-constant rational
function on $Y$. Let $W$ be a subvariety of $Y$. Denote for $n\in
\mathbb{Z}$ by $\mu(\mathfrak{X}_{n}^W)$ the motivic measure of the
arcs $\gamma $ on $Y$ with origin in $W$ for which $ord_{t}(
\gamma^{\ast}(f/g)) =n$.
 A priori, it is not clear that this set of
arcs is measurable, but one can show that $\mu(\mathfrak{X}_{n}^W)$
is well defined in $\mathcal{M}$ (see \cite[Theorem 10.1.1]{CL1}).

\begin{definition}
\label{defmotzeta}\rm With notation as above, the \textit{motivic
zeta function} associated to $f/g$ (and $W$) is
\begin{equation*}
Z_{W}(f/g);T) =\sum\limits_{n\in\mathbb{Z}}\ \mu (\mathfrak{X}
_{n}^W) T^n.
\end{equation*}
\end{definition}

\noindent
 Note that this is a series in $T$ and $T^{-1}$ (over
$\mathcal{M}$), and that such expressions in general do not form a
ring. As in the classical case however, $Z_{W}(f/g);T)$ turns out to
be a rational function in $T$, see \cite[Theorem 5.7.1]{CL1}.

Using the change of variables formula in motivic integration
(\cite[Lemma 3.3]{D-L2},\cite[Theorem 12.1.1]{CL1}), one obtains the
following
 formula for $Z_{W}(f/g);T)$  in terms of an embedded
resolution of $f^{-1}\{0\}\cup g^{-1}\{0\}$, which generalizes the
similar formula in the classical case.

%\begin{definition}
%\label{defmotzeta}\rm With notation as above, the \textit{motivic
%zeta function} associated to $(f,g)$ (and $W$) is the formal power
%series in two variables
%\begin{equation*}
%Z_{W}\big((f,g);T,U\big) =\sum\limits_{i,j\geq 0}\ \mu (\mathfrak{X}_{i,j}^W) T^iU^j.
%\end{equation*}
%\end{definition}

%\begin{theorem}
%\label{Theoremmotivic2} Let $\sigma :X\rightarrow Y$ be an embedded
%resolution of $f^{-1}\{0\}\cup g^{-1}\{0\}$, for which we use the
%analogous notation $v_{i}$, $E_{i}$, $N_{f,i}$ and $N_{g,i}$ $(i\in
%T)$ as in Theorem \ref{thresolsing}. With also $E_{I}^{\circ
%}:=\left( \cap _{i\in I}E_{i}\right) \setminus \left( \cup _{k\notin
%I}E_{k}\right) $ for $ I\subset T$, we have
%\begin{equation*}
%Z_{W}\big( (f,g);T,U\big) =\sum\limits_{I\subset T}\left[
%E_{I}^{\circ }\cap \sigma ^{-1}W\right] \prod\limits_{i\in
%I}\frac{\left( \mathbb{L}-1\right)
%T^{N_{f,i}}U^{N_{g,i}}}{\mathbb{L}^{v_{i}}-T^{N_{f,i}}U^{N_{g,i}}}.
%\end{equation*}
%In particular $Z_{W}\big( (f,g);T,U\big)$ is rational in $T$ and
%$U$.
%\end{theorem}
%
%\begin{definition}\rm
%$Z_{W}\left( f/g;T\right) = Z_{W}\big( (f,g);T,T^{-1}\big)$
%\end{definition}

\begin{theorem}
\label{Theoremmotivic1} Let $\sigma :X\rightarrow Y$ be an embedded
resolution of $f^{-1}\{0\}\cup g^{-1}\{0\}$, for which we use the
analogous notation $E_{i}$, $N_i= N_{f,i} - N_{g,i}$ and $v_{i}$
$(i\in T)$ as in Theorem \ref{thresolsing}. With also $E_{I}^{\circ
}:=\left( \cap _{i\in I}E_{i}\right) \setminus \left( \cup _{k\notin
I}E_{k}\right) $ for $ I\subset T$, we have
\begin{equation*}
Z_{W}\left( f/g;T\right) =\sum\limits_{I\subset T}\left[
E_{I}^{\circ }\cap \sigma ^{-1}W\right] \prod\limits_{i\in
I}\frac{\left( \mathbb{L}-1\right)
T^{N_{i}}}{\mathbb{L}^{v_{i}}-T^{N_{i}}}.
\end{equation*}
\end{theorem}

Note that above we use only very special cases of the far reaching
general theory in \cite{CL1}. Alternatively, one could obtain
Theorem \ref{Theoremmotivic1} by considering a classical
two-variable motivic zeta function associated to $f$ and $g$, the
change of variables formula, and a specialization argument as in the
proof of Theorem \ref{Theorem1}.

We also want to mention that Raibaut \cite[Section 4.1]{R}
introduced another kind of motivic zeta function for a rational
function, in order to define a motivic Milnor fibre at the germ of
an indeterminacy point.
\bigskip

Specializing to topological Euler characteristics, denoted by $\chi
\left( \cdot \right) $, as in \cite[(2.3)]{D-L} or \cite[(6.6)]{V1},
we obtain the expression
\begin{equation*}
Z_{top,W}\left(f/g;s\right) :=\sum_{I\subset T}\chi \left(
E_{I}^{\circ }\cap \sigma ^{-1}W\right) \prod_{i\in I}\frac{1}{
v_{i}+N_{i}s}\in \mathbb{Q}(s),
\end{equation*}
which is then independent of the chosen embedded resolution. (When
the base field is not the complex numbers, we consider $\chi (\cdot
)$ in \'{e}tale $\overline{\mathbb{Q}}_\ell$-cohomology as in
\cite{D-L}.) It can be taken as a definition for the
\textit{topological zeta function} associated to $f/g$ (and $W$),
generalizing the original one of Denef and Loeser associated to a
polynomial \cite{D-L3}.

For $n=2$, Gonzalez Villa and Lemahieu give in \cite{GL} a different
definition of a topological zeta function associated to a
meromorphic function $f/g$, roughly not taking into account the
components $E_i$ with $N_i<0$. That construction is more ad hoc, but
well suited to generalize the so-called {\em monodromy conjecture}
in dimension $2$ to meromorphic functions.

\bigskip

\section{Oscillatory Integrals}\label{section oscillatory}

In this section we start the study of the asymptotic behavior of
oscillatory integrals attached to meromorphic functions. In the next
section we then prove our main results.

\subsection{The set-up}

\subsubsection{Additive characters}

%We denote $e\left( \cdot\right) :=\exp\left(
%2\pi\sqrt{-1}\cdot\right) $(FATTER SYMBOL?), and set
We denote
\[
\Psi\left(  x\right)  :=\left\{
\begin{array}
[c]{lll}
\exp\left(2\pi\sqrt{-1}x\right) & \text{if} & x\in K=\mathbb{R}\\
&  & \\
\exp\left(4\pi\sqrt{-1}\operatorname{Re}(x)\right) & \text{if} &
x\in K=\mathbb{C}\text{.}
\end{array}
\right.
\]
We call $\Psi$ \textit{the standard additive character on}
$K=\mathbb{R}$, $\mathbb{C}$. Given
\[
z=\sum_{n=n_{0}}^{\infty}z_{n}p^{n}\in\mathbb{Q}_{p}\text{, with }z_{n}%
\in\left\{  0,\ldots,p-1\right\}  \text{ and }z_{n_{0}}\neq0,
\]
we set
\[
\left\{  z\right\}  _{p}:=\left\{
\begin{array}
[c]{lll}%
0 & \text{if} & n_{0}\geq0\\
&  & \\
\sum_{n=n_{0}}^{-1}z_{n}p^{n} & \text{if} & n_{0}<0,
\end{array}
\right.
\]
\textit{the fractional part of }$z$. We also put $\{0\}_p=0$. Then
$\exp(2\pi\sqrt{-1}\left\{ z\right\} _{p}),$ $z\in\mathbb{Q}_{p}$,
determines an additive character on $\mathbb{Q}_{p}$, trivial on
$\mathbb{Z}_{p}$ but not on $p^{-1}\mathbb{Z}_{p}$.

For a finite extension $K$ of $\mathbb{Q}_{p}$, we recall that there
exists an integer $d\geq0$ such that $Tr_{K/\mathbb{Q}
_{p}}(z)\in\mathbb{Z}_{p}$ for $\left\vert z\right\vert _{K}\leq
q^{d}$, but $Tr_{K/\mathbb{Q}_{p}}(z_{0})\notin\mathbb{Z}_{p}$ for
some $z_{0}$\ with $\left\vert z_{0}\right\vert _{K}=q^{d+1}$. The
integer $d$ is called \textit{the exponent of the different} of
$K/\mathbb{Q}_{p}$. It is known that $d\geq e-1$, where $e$ is the
ramification index of $K/\mathbb{Q}_{p}$, see e.g. \cite[VIII
Corollary of Proposition 1]{We}. The additive character
\[
\Psi(z):=\exp\Big(2\pi\sqrt{-1}\left\{  Tr_{K/\mathbb{Q}_{p}}(\mathfrak{p}%
^{-d}z)\right\}  _{p}\Big),\text{ }z\in K\text{, }%
\]
is a \textit{standard additive character} of $K$, i.e., $\Psi$ is
trivial on $R_{K}$ but not on $P_{K}^{-1}$.
% For our purposes, it is more convenient to use
%\[
%\Psi(z)=\exp(2\pi\sqrt{-1}\left\{  Tr_{K/\mathbb{Q}_{p}}(z)\right\}
%%_{p}),\text{ }z\in K\text{, }%
%\]
%instead of $\varkappa(\cdot)$. This particular choice is due to the
%fact that we use Denef's approach for estimating oscillatory
%integrals, see \cite[Proposition 1.4.4]{D0}. MAYBE NOT???

\bigskip
\subsubsection{Oscillatory integrals}
 We take $\frac{f}{g}:U\to K$ and $\Phi$ as in Section \ref{section preliminaries}. The
oscillatory integral attached to $\left( \frac{f}{g},\Phi\right) $
is defined by
\[
E_{\Phi}\left(  z;f/g\right)  =%
{\displaystyle\int\limits_{U\smallsetminus D_{K}}} \Phi\left(
x\right) \Psi\left( z\cdot\frac{f\left( x\right) }{g\left(
x\right)  }\right)  \left\vert dx\right\vert _{K}%
\]
for $z\in K$. Our goal is to study the asymptotic behavior of
$E_{\Phi}\left( z;f/g\right) $ as $\left\vert z\right\vert
_{K}\rightarrow+\infty$ and of $E_{\Phi }\left( z;f/g\right)
-\int_{U}\Phi\left( x\right) \left\vert dx\right\vert _{K}$ as
$\left\vert z\right\vert _{K}\rightarrow0$.

%It is related to the fibre integral (IS THE FIBRE CONTAINED IN
%$U\setminus D_K$?)
%$$ F_{\Phi}\left( t;f/g\right) = \int_{\frac fg\left( x\right) =t}
%\Phi\left( x\right) \left\vert \frac{dx}{d(f/g)}\right\vert_{K},
%$$
%for $t\in K^{\times}$, where $\frac{dx}{d(f/g)}$ denotes the
%Gel'fand-Leray differential form on $\frac fg\left( x\right) =t$.

% and the relation with expansions of fibre integrals.

\subsection{Classical results}

In the classical case of an analytic function $f$ (when $g=1$), the
oscillatory integral above is related to the zeta function
$Z_\Phi(\omega;f)$ and to the fibre integral
$$F_{\Phi}\left( t;f\right)=\int_{f\left( x\right) =t}
\Phi\left( x\right) \left\vert \frac{dx}{df}\right\vert_{K}$$ for
$t\in K^{\times}$, where $\frac{dx}{df}$ denotes the Gel'fand-Leray
differential form on $f\left( x\right) =t$. More precisely, denoting
by $C_f:=\{z\in U \mid \nabla f(z)=0\}$ the critical set of $f$, one
assumes that $C_f\cap \supp \Phi \subset f^{-1}\{0\}$. Then
asymptotic expansions of $E_{\Phi}\left( z;f\right)$ as $\left\vert
z\right\vert _{K}\rightarrow+\infty$ are closely related to
asymptotic expansions of $F_{\Phi}\left( t;f\right)$ as $\left\vert
t\right\vert_{K}\rightarrow0$, and to poles of $Z_\Phi(\omega;f)$.
These matters were extensively studied by many authors, for instance
in several papers of Barlet in the case of $\mathbb{R}$-fields. See
also \cite{AVG} when $K=\mathbb{R}$. Igusa presented in \cite{I3} a
uniform theory over arbitrary local fields of characteristic zero
(for polynomials $f$, but also valid for analytic functions).

We first recall three spaces of functions from \cite{I3} that we
will use, and the relations between them.

\subsubsection{Relating spaces of functions through Mellin and Fourier transform}

\begin{definition}\label{functionclassreal}\rm   Let $K$ be an $\mathbb{R}$-field.
Let $\Lambda$ be a strictly increasing (countable) sequence of
positive real numbers with no finite accumulation point, and $\{
n_\lambda\}_{\lambda\in\Lambda}$ a sequence of positive integers.

\smallskip

(1) We define $\mathcal {G}$ as the space of all complex-valued
functions $G$ on $K^\times$ such that
\smallskip
\noindent (i) $G\in C^{\infty }\left( K^{\times}\right) $;

\noindent (ii) $G(x)$ behaves like a Schwartz function of $x$ as
$\left\vert x\right\vert _{K}$ tends to infinity;

\noindent (iii) we have the asymptotic expansion
\begin{equation}
G\left( x\right) \approx \sum_{\lambda\in\Lambda}
\sum_{m=1}^{n_\lambda} b_{\lambda,m}\left( \frac{x}{\left\vert
x\right\vert }\right) \left\vert x\right\vert _{K}^{\lambda-1}\left(
\ln\left\vert x\right\vert _{K}\right) ^{m-1}\text{ as }\left\vert
x\right\vert _{K}\rightarrow0,\label{Gdef}
\end{equation}
where the $b_{\lambda,m}$ are smooth functions on $\left\{ u\in
K^{\times}\mid \left\vert u\right\vert _{K}=1\right\} $. In
addition, this expansion is termwise differentiable and uniform in
$x/|x|$ (we refer to \cite[page 23]{I3} for a detailed explanation
of this terminology).

\smallskip (2) We define $\mathcal {Z}$ as the space of all
complex-valued functions $Z$ on $\Omega(K^\times)$ such that
\smallskip

\noindent (i) $Z(\omega=\omega_s\chi_l(ac))$ is meromorphic on
$\Omega(K^\times)$ with poles at most at $s\in -\Lambda$;

\noindent (ii) there are constants $r_{\lambda,m,l}$ such that
$Z(\omega_s\chi_l(ac))-\sum_{m=1}^{n_\lambda}\frac{r_{\lambda,m,l}}{(s+\lambda)^m}$
is holomorphic for $s$ close enough to $-\lambda$, for all $\lambda
\in \Lambda$;

\noindent (iii) for every polynomial $P\in \mathbb{C}[s,l]$ and
every vertical strip in $\mathbb{C}$ of the form
$B_{\sigma_{1},\sigma_{2}}:=\left\{ s\in \mathbb{C} \mid
\operatorname{Re}(s)\in\left( \sigma_{1},\sigma_{2}\right) \right\}
$, we have that $P( s,l) Z (\omega_s\chi_l(ac))$ is uniformly
bounded for all $l$ and for $s\in B_{\sigma_{1},\sigma_{2}}$ with
neighborhoods of the points in $-\Lambda$ removed therefrom.
\end{definition}

\begin{definition}\label{functionclasspadic}\rm
 Let $K$ be a $p$-field. Let $\Lambda$ be a finite set of complex
numbers$\mod 2\pi i / \ln q$ with positive real part, and $\{
n_\lambda\}_{\lambda \in \Lambda}$ a sequence of positive integers.

\smallskip

(1) We define $\mathcal {G}$ as the space of all complex-valued
functions $G$ on $K^\times$ such that
\smallskip
\noindent (i) $G$ is locally constant;

\noindent (ii) $G(x)=0 $ when $\left\vert x\right\vert _{K}$ is
sufficiently large;

\noindent (iii) we have the expansion
\begin{equation}
G\left( x\right) = \sum_{\lambda\in \Lambda} \sum_{m=1}^{n_\lambda}
b_{\lambda,m}\left( ac(x)\right) \left\vert x\right\vert
_{K}^{\lambda-1}\left( \ln\left\vert x\right\vert _{K}\right)
^{m-1}\text{ for sufficiently small}\left\vert x\right\vert
_{K},\label{Gdefp}
\end{equation}
where the $b_{\lambda,m}$ are locally constant functions on
$R_K^\times$.

\smallskip (2) We define $\mathcal {Z}$ as the space of all
complex-valued functions $Z$ on $\Omega(K^\times)$ such that
\smallskip

\noindent (i) there are constants $r_{\lambda,m,\chi}$ such that
$Z(\omega_s\chi(ac))-\sum_{\lambda\in
\Lambda}\sum_{m=1}^{n_\lambda}\frac{r_{\lambda,m,\chi}}{(1-q^{-\lambda}q^{-s})^m}$
is a Laurent polynomial in $q^{-s}$;

\noindent (ii) %$Z_\chi(z)$ is identically zero for almost all
 there exists a positive integer $e$  such that
$Z (\omega_s\chi(ac))$ is identically zero unless the conductor
$e_{\chi}$ of $\chi$ satisfies $e_{\chi}\leq e $.
\end{definition}

In order to simplify notation in the sequel of this article, we
adapted slightly Igusa's setup in the sense that we consider
expansions involving $\left\vert x\right\vert _{K}^{\lambda-1}$
instead of $\left\vert x\right\vert _{K}^{\lambda}$ in (\ref{Gdef})
and (\ref{Gdefp}). This implies that the statements of the theorems
below differ also slightly from the formulation in \cite{I3}.

\begin{theorem}{\cite[I  Theorems 4.2, 4.3 and 5.3]{I3}} \label{Thm_Igusa_ZF}
The Mellin transform $M_K$ induces a bijective correspondence
between the spaces $|x|_K\mathcal {G}$ and $\mathcal {Z}$. More
precisely, for any $G\in \mathcal {G}$, define $M_K (|x|_KG)$ as a
function on $\Omega_{(0,+\infty)}(K^\times)$ by
$$
M_K(|x|_K G)(\omega)=\int_{K^\times}|x|_K G(x)\omega(x) \frac
{|dx|_K}{m_K|x|_K},
$$
where $m_K$ is $2$, $2\pi$ or $1-q^{-1}$, according as $K$ is
$\mathbb{R}$, $\mathbb{C}$, or a $p$-field. Then its meromorphic
continuation to $\Omega(K^\times)$ is in $\mathcal {Z}$, and
%Here $|d^\times x|_K$ is $\frac {|dx|_K}{2|x|_K}$, $\frac
%{|dx|_K}{2\pi|x|_K}$ or $\frac {|dx|_K}{(1-q^{-1})|x|_K}$, according
%as $K$ is $\mathbb{R}$, $\mathbb{C}$ or a $p$-field.
the map $G\mapsto M_K (|x|_K G)$ induces a bijection between the
spaces $|x|_K \mathcal {G}$ and $\mathcal {Z}$.
\end{theorem}

\begin{remark}\rm
Under the bijective correspondence $|x|_KG\mapsto Z$ above, the
coefficients in the expansion of $|x|_KG$ and the coefficients in
the Laurent expansion of $Z$ around its poles determine each other.
See \cite[I Theorems 4.2, 4.3 and 5.3]{I3} for the exact relations.
\end{remark}

\begin{theorem}{\cite[II Theorem 2.1]{I3}} \label{Thm_Igusa_2}
(1) Let $K$ be an $\mathbb{R}$-field. Let $\Lambda$ be a strictly
increasing (countable) sequence of positive real numbers with no
finite accumulation point, and $\{ n_\lambda\}_{\lambda\in\Lambda}$
a sequence of positive integers. Then the space $\mathcal {G}^*$ of
Fourier transforms of functions in $\mathcal {G}$ consists precisely
of all complex-valued functions $H$ on $K$ such that

\smallskip
\noindent (i) $H\in C^{\infty }\left( K\right) $;

\noindent (ii) we have the termwise differentiable uniform asymptotic
 expansion
\begin{equation}H\left( x\right)\approx \sum_{\lambda\in\Lambda} \sum_{m=1}^{n_\lambda}
c_{\lambda,m}\left( \frac{x}{\left\vert x\right\vert }\right)
\left\vert x\right\vert _{K}^{-\lambda}\left( \ln\left\vert
x\right\vert _{K}\right) ^{m-1}\text{ as }\left\vert x\right\vert
_{K}\rightarrow \infty,\label{Hthm}
\end{equation}
where the $c_{\lambda,m}$ are smooth functions on $\left\{ u\in
K^{\times}\mid \left\vert u\right\vert _{K}=1\right\} $.
Furthermore, if $H$ is the Fourier transform $G^*$ of $G\in\mathcal
{G}$ with asymptotic expansion (\ref{Gdef}),
%\begin{equation}
%G\left( x\right) \approx \sum_{\lambda\in\Lambda}
%\sum_{m=1}^{n_\lambda} b_{\lambda,m}\left( \frac{x}{\left\vert
%x\right\vert }\right) \left\vert x\right\vert _{K}^{\lambda-1}\left(
%\ln\left\vert x\right\vert _{K}\right) ^{m-1}\text{ as }\left\vert
%x\right\vert _{K}\rightarrow 0,\label{Gthm}
%\end{equation}
then (\ref{Hthm}) is the termwise Fourier transform of (\ref{Gdef}).

(2) Let $K$ be a $p$-field. Let $\Lambda$ be a finite set of complex
numbers$\mod 2\pi i / \ln q$ with positive real part, and $\{
n_\lambda\}_{\lambda \in \Lambda}$ a sequence of positive integers.
Then the space $\mathcal {G}^*$ of Fourier transforms of functions
in $\mathcal {G}$ consists precisely of all complex-valued functions
$H$ on $K$ such that
\smallskip

\noindent (i) $H$ is locally constant;

\noindent (ii) we have the expansion
\begin{equation}
H\left( x\right) = \sum_{\lambda\in \Lambda} \sum_{m=1}^{n_\lambda}
c_{\lambda,m}\left( ac(x)\right) \left\vert x\right\vert
_{K}^{-\lambda}\left( \ln\left\vert x\right\vert _{K}\right)
^{m-1}\text{ for sufficiently large}\left\vert x\right\vert
_{K},\label{Hthmp}
\end{equation}
where the $c_{\lambda,m}$ are locally constant functions on
$R_K^\times$. Furthermore, if $H$ is the Fourier transform $G^*$ of
$G\in\mathcal {G}$ satisfying (\ref{Gdefp}),
%\begin{equation}
%G\left( x\right) = \sum_{\lambda\in \Lambda} \sum_{m=1}^{n_\lambda}
%b_{\lambda,m}\left( ac(x)\right) \left\vert x\right\vert
%_{K}^{\lambda-1}\left( \ln\left\vert x\right\vert _{K}\right)
%^{m-1}\text{ for sufficiently small }\left\vert x\right\vert
%_{K},\label{Gthmp}
%\end{equation}
then (\ref{Hthmp}) is the termwise Fourier transform of
(\ref{Gdefp}).

(3) For all $K$ the inverse of the map $\mathcal {G}\to\mathcal
{G}^*:G\mapsto G^*$ is given by the \lq generalized Fourier
transform\rq\ $H\mapsto H^*(-x)$, where
$$H^*(x)=\lim_{r\to \infty} \int_{|y|_K \leq r} H(y)\Psi(xy)\, |dy|_K
.$$
\end{theorem}

In order to describe the relation between the coefficients $b$ and
$c$ above, we first introduce some notation. Let $\chi$ be a
character of $\left\{ u\in K^{\times}\mid \left\vert u\right\vert
_{K}=1\right\} $.
%When $K$ is an $\mathbb{R}$-field, we denote $d_K:=1$ if $K=\mathbb{R}$ and $d_K:=1/2$ if $K=\mathbb{C}$.
We first assume that $K$ is an $\mathbb{R}$-field. For
$\chi=\chi_l$, we denote
\[
w_{\chi}\left( s\right) =i^{\left\vert l\right\vert }\left(
\pi\left[ K:\mathbb{R}\right] \right) ^{\frac{1}{2}\left[
K:\mathbb{R}\right]
\left(  1-2s\right)  }\frac{\Gamma\left(  \frac{1}{2}\left[  K:\mathbb{R}%
\right] s+\frac{\left\vert l\right\vert }{2}\right) }{\Gamma\left(
\frac {1}{2}\left[ K:\mathbb{R}\right] \left( 1-s\right)
+\frac{\left\vert l\right\vert }{2}\right) }\, ,
\]
where $\Gamma(s)$ is the complex Gamma function. When $K$ is a
$p$-field, we put
\[
w_\chi(s)= \frac{1-q^{s-1}}{1-q^{-s}} \quad\text{if } \chi=1
\qquad\text{ and }\qquad w_\chi(s)= g_\chi q^e \chi^s \quad\text{if
} \chi\neq 1 ,
\]
where $e$ is the conductor of $\chi$ and $g_{\chi}$ is a (non-zero)
Gaussian sum (see \cite[page 57]{I3}).

\begin{proposition}{\cite[II Remark 1 page 67]{I3}}\label{Prop Igusa} (1) For all $\lambda$ and $m$, let $b_{\lambda,m,\chi}$ be the Fourier coefficients
of $b_{\lambda,m}\left( u\right) $ in Definitions
\ref{functionclassreal} and \ref{functionclasspadic}, that is, we
have the expansion
\[
b_{\lambda,m}(u) = \sum_{\chi} b_{\lambda,m,\chi}\chi(u),
\]
and similarly
\[
c_{\lambda,m}(u) = \sum_{\chi} c_{\lambda,m,\chi}\chi(u)
\]
for $c_{\lambda,m}(u)$ in Theorem \ref{Thm_Igusa_2}. Here, when $K$
is an $\mathbb{R}$-field, $\chi$ runs over the $\chi_l$ with
$l\in\{0,1\}$ or $l\in \mathbb{Z}$, according as $K$ is $\mathbb{R}$
or $\mathbb{C}$, and when $K$ is a $p$-field, $\chi$ runs over the
characters of $R_K^\times$.

These Fourier coefficients are related as follows:
\begin{equation}\label{relatebtoc}
c_{\lambda,m,\chi^{-1}} = \left( -1\right) ^{m-1}
{\displaystyle\sum\limits_{j=m}^{n_{\lambda}}} \left(
\begin{array}
[c]{c}
j-1\\
m-1
\end{array}
\right)\left( \frac{d^{j-m}}{ds^{j-m}}w_{\chi}\left( \lambda\right)
\right) b_{\lambda,j,\chi}.
\end{equation}
%for every $\lambda$ and $m$.
% Here the index $l$ is restricted to $\left\{ 0,1\right\} $ or runs over the whole $\mathbb{Z}$, according as $K$ is $\mathbb{R}$ or $\mathbb{C}$.
Fixing $\lambda$ and $\chi$, the equation (\ref{relatebtoc})
expresses, for $1\leq m \leq n_\lambda$, the
$c_{\lambda,m,\chi^{-1}}$ linearly in terms of the
$b_{\lambda,m,\chi}$ with a coefficient matrix that is upper
triangular of size $n_\lambda$, with everywhere $w_{\chi}(\lambda)$
on the diagonal, and nonzero multiples of $w'_{\chi}(\lambda)$ just
above the diagonal.

(2) If $w_{\chi}(\lambda)\neq 0$, then the rank of that matrix is
$n_\lambda$ and then the $b_{\lambda,m,\chi}$ can conversely be
expressed in and are determined by the $c_{\lambda,m,\chi^{-1}}$.

On the other hand, if $w_{\chi}(\lambda)= 0$, then
$w'_{\chi}(\lambda)\neq 0$, and hence the rank of that matrix is
$n_\lambda-1$. In this case, $c_{\lambda,n_\lambda,\chi^{-1}}=0$,
the coefficient $b_{\lambda,1,\chi}$ does not appear in the
expression for the $c_{\lambda,m,\chi^{-1}}, 1\leq m \leq n_\lambda
-1$, and the $b_{\lambda,m,\chi}, 2\leq m \leq n_\lambda,$ can be
expressed in and are determined by the $c_{\lambda,m,\chi^{-1}},
1\leq m \leq n_\lambda -1$.

(3) For an $\mathbb{R}$-field $K$ and $\chi=\chi_l$, we have that
$w_{\chi}(s)= 0$ if and only if $s\in 1+\frac{1}{[K:\mathbb{R}]}(
|l| + 2\mathbb{Z}_{\geq 0})$. For a $p$-field $K$, we have that
$w_{\chi}(s)= 0$ if and only if $\chi=1$ and $s= 1 \mod 2\pi i / \ln
q$.

%ADAPT? This implies that $c_{\lambda,m} =0$ for $m=m_\lambda$ when
%$\lambda$ or $2\lambda$ is an integer, according as $K$ is
%$\mathbb{R}$ or $\mathbb{C}$.
\end{proposition}

\begin{remark}\label{conductor}\rm
In the classical case of a polynomial/analytic function $f$ (when
$g=1$), an important result in \cite{I3} is that $Z_\Phi(\omega;f)$
belongs to the space $\mathcal {Z}$, assuming that $C_f \cap \supp
\Phi \subset f^{-1}\{0\}$. Then Theorems \ref{Thm_Igusa_ZF} and
\ref{Thm_Igusa_2} lead to the relations between asymptotic
expansions of $E_{\Phi}\left( z;f\right)$ as $\left\vert
z\right\vert _{K}\rightarrow+\infty$, asymptotic expansions of
$F_{\Phi}\left( t;f\right)$ as $\left\vert
t\right\vert_{K}\rightarrow0$, and poles of $Z_\Phi(\omega;f)$.

\smallskip
Now for meromorphic functions $f/g$, the zeta function
$Z_\Phi(\omega;f/g)$ is in general {\em not} in $\mathcal {Z}$, and
restricting the support of $\Phi$ will not remedy that. As an
illustration we indicate the problem when $K$ is a $p$-field: in
general
 $Z_{\Phi}\left( \omega;f/g\right)$ is {\em not} identically zero when the conductor
 of $\chi=\omega \mid_{R_{K}^{\times}}$ is large enough.

A crucial point in the proof of this result in the classical case is
namely that an integral of the form $ \int_{p^e\mathbb{Z}_p}
\chi(1+x) \, |dx|_{\mathbb{Q}_p}$ vanishes as soon as the conductor
of $\chi$ is large enough. In our case, due to the new feature in
Remark \ref{units}, one also has to deal for instance with integrals
of the form
$$\int_{p^e\mathbb{Z}_p} \chi(1+x) \, |x|_{\mathbb{Q}_p}^{\nu-1} \, |dx|_{\mathbb{Q}_p} $$
with $\nu > 1$. And those do {\em not} vanish for $\chi$ with large
enough conductor, as is shown by a straightforward computation.

\smallskip
 In order to derive asymptotic expansions for
$E_{\Phi}\left( z;f/g\right)$, we need a new strategy. We do use
Theorems \ref{Thm_Igusa_ZF} and \ref{Thm_Igusa_2} in the sequel, but
in an appropriate different setting.
\end{remark}

%\medskip
%MAYBE LATER SOMEWHERE AN EXAMPLE FOR R AND p-ADIC FOR:
% For a general analytic (or even polynomial) function $h$ that is NOT a monomial, the
%usual statements of belonging to the three spaces of functions are
%{\em not} valid for arbitrary $v_i$.
%MAYBE IT IS OK FOR C ??

\subsection{Set-up on the resolution space}\label{subsection set-up resolution}

From now on we fix an adapted embedded resolution $\sigma:X_K\to U$
of $D_K$, obtained by resolving the indeterminacies of the map
$f/g$, we use the notation of Theorem \ref{thresolsing} and the
classical notation $\sigma^* \Phi=\Phi \circ \sigma$. We denote by
$\rho: X_K\to \mathbb{P}^1$ the {\em morphism} $(f/g) \circ \sigma$.

Analogously as in Remarks \ref{altern} and \ref{alternreal}, we can
describe
$$E_{\Phi}\left(
z;f/g\right)= \int_{X_K\setminus \sigma^{-1} D_K} (\sigma^*
\Phi)(y)\, \Psi \big(z\cdot \rho(y)\big)\, |\sigma^* dx|_K$$
 as a complex linear combination of local contributions, described
 by \lq elementary\rq\ oscillating integrals in local coordinates
 around appropriately chosen points $b\in X_K$. We first stress the
 following fact.

 \begin{remark}\rm
 A point $b \in X_K$ for
which $N_i=N_{f,i}-N_{g,i}>0$ for at least one $i \in
\{1,\dots,r\}$, satisfies $\rho(b)=0$. Similarly, when $N_i<0$ for
at least one $i$, we have that $\rho(b)=\infty$. And if all $N_i=0$,
then $\rho(b)\in K^\times$. As already mentioned in Remark
\ref{units}, this last case also occurs for points $b$ in
$\sigma^{-1}(D_K)$.
\end{remark}

These three possibilities for $\rho(b)$ yield the following three
types of local contributions. We use local coordinates
$y_1,\dots,y_n$ around $b$, and
 $\Theta$ denotes  the characteristic function of a small
polydisc $B$ around the origin, or a smooth function supported in
the polydisc
\[
B=\left\{ y\in K^{n}\mid\left\vert y_{j}\right\vert_K <1\text{ for
}j=1,\ldots ,n\right\} ,
\]
according as $K$ is a $p$-field or an $\mathbb{R}$-field. When
$\rho(b)=0$ and $\rho(b)=\infty$, we have respectively the types

\begin{equation}
E_{1}\left(  z\right)  :=%
{\displaystyle\int\limits_{K^{n}}} \Theta\left( y\right) \Psi\left(
z \cdot c {\displaystyle\prod_{i=1}^r} y_{i}^{M_{i}}\right) \left(
{\displaystyle\prod_{i=1}^n} \left\vert y_{i}\right\vert_K
^{v_{i}-1}\right) \left\vert dy\right\vert _{K}
\label{E_1}%
\end{equation}
and
\begin{equation}
E_{2}\left( z\right) :=
{\displaystyle\int\limits_{K^{n}\smallsetminus\cup_{i=1}^r\left\{
y_{i}=0\right\} }} \Theta\left( y\right) \Psi\left( z \cdot (c
{\displaystyle\prod_{i=1}^r} y_{i}^{M_{i}})^{-1}\right) \left(
{\displaystyle\prod_{i=1}^n} \left\vert y_{i}\right\vert_K
^{v_{i}-1}\right) \left\vert dy\right\vert _{K},
\label{E_2}%
\end{equation}
where $r$, all $M_i$ and all $v_i$ are positive integers (and $r\leq
n$), and $c$ is a nonzero constant (that can be taken to be $1$ if
$K=\mathbb{C}$, and $1$ or $-1$ if $K=\mathbb{R}$). Note that
possibly some $v_i=1$. When $\rho(b)\in K^\times$, we have the type
\begin{equation}
E_{3}\left( z\right) := {\displaystyle\int\limits_{K^{n}}}
\Theta\left( y\right) \Psi\big( z \cdot u(y) \big) \left(
{\displaystyle\prod_{i=1}^n} \left\vert y_{i}\right\vert_K
^{v_{i}-1}\right)
 \left\vert
dy\right\vert _{K}, \label{E_3}
\end{equation}
where $u(y)$ is invertible on the support of $\Theta$.

\subsection{Some auxiliary expansions}\label{subsection auxexpansions}

In order to study
$$E_{\Phi}\left(
z;f/g\right)= \int_{X_K\setminus \sigma^{-1} D_K} (\sigma^* \Phi)(y)
\, \Psi \big(z\cdot \rho(y)\big)\, |\sigma^* dx|_K$$
 via local contributions,
we now fix an appropriate decomposition of $\sigma^* \Phi$.

 We cover
the (compact) support of $\sigma^* \Phi$ by finitely many polydiscs
$B_j, j\in J,$ as above (all disjoint in the case of $p$-fields),
making sure that the (compact) fibres $\rho^{-1}\{0\}\cap \supp
(\sigma^* \Phi)$ and $\rho^{-1}\{\infty\}\cap \supp (\sigma^* \Phi)$
are completely covered by certain $B_j, j\in J_0 \subset J,$ and
$B_j, j\in J_\infty \subset J,$ with centre $b_j$ mapped by $\rho$
to $0$ and $\infty$, respectively.
 When $K$ is a $p$-field, we define $\Phi_0$ and $\Phi_\infty$
as the restriction of $\sigma^* \Phi$ to $\cup_{j\in J_0}B_j$ and
$\cup_{j\in J_\infty}B_j$, respectively. When $K$ is an
$\mathbb{R}$-field, $\Phi_0$ and $\Phi_\infty$ are these
restrictions, modified by the partition of the unity that induces
the functions $\Theta$ on the local charts $B_j$.

We now define
$$E_{0}\left(
z\right)= \int_{X_K\setminus \sigma^{-1} D_K} (\Phi_0)(y) \,\Psi
\big(z\cdot \rho(y)\big) \,|\sigma^* dx|_K$$ and
$$E_{\infty}\left(
z\right)= \int_{X_K\setminus \sigma^{-1} D_K} (\Phi_\infty)(y)\,
\Psi \big(z\cdot \rho(y)\big)\, |\sigma^* dx|_K .$$

\bigskip
 We first treat the case of oscillatory integrals over $\mathbb{R}$-fields.

\begin{proposition}\label{propE0}  Let $K$ be an $\mathbb{R}$-field.
 We denote by
$m_{\lambda}$ the order of a pole $\lambda $ of $Z_{\Phi}\left(
\omega;f/g \right) $.

(1) Then $E_{0}\left( z\right) $ has an asymptotic expansion of the
form
\begin{equation}
E_{0}\left(  z\right) \approx%
{\displaystyle\sum\limits_{\lambda <0}}\
{\displaystyle\sum\limits_{m=1}^{m_{\lambda}}} A_{\lambda,m}\left(
\frac{z}{\left\vert z\right\vert }\right) \left\vert z\right\vert
_{K}^{\lambda}\left( \ln\left\vert z\right\vert _{K}\right)
^{m-1}\text{ as }\left\vert z\right\vert _{K}\rightarrow \infty,
\label{AsimE0}
\end{equation}
where $\lambda$ runs through all the negative poles of
$Z_{\Phi}\left( \omega;f/g \right) $ (for all $\omega$), and each
$A_{\lambda,m}$ is a smooth function on $\left\{ u\in K^{\times}\mid
\left\vert u\right\vert _{K}=1\right\}$.

(2) Writing $\omega=\omega_s\chi_l(ac)$, we have that
$A_{\lambda,m_\lambda}=0$ if $|\lambda| \in
1+\frac{1}{[K:\mathbb{R}]}( |l| + 2\mathbb{Z}_{\geq 0})$.
\end{proposition}

\begin{proof}
(1) We apply the ideas of Igusa's theory in \cite[III]{I3} to the
function $\rho$ on $X_K$ (in fact rather to the $K$-valued
restriction of $\rho$ to $X_K\setminus \rho^{-1}\{\infty\}$) and the
smooth function with compact support $\Phi_0$. Actually, Igusa
formulates everything for polynomial functions, but his results are
also valid for analytic functions on smooth manifolds, since his
arguments are locally analytic on an embedded resolution space.

A crucial condition for his arguments to work is that $C_\rho \cap
\operatorname{supp} \Phi_0 \subset \rho^{-1}\{0\}$, where $C_\rho$
denotes the critical locus of $\rho$. This condition is
%clearly
satisfied by our choice of $\Phi_0$ and because $\rho$ is locally
monomial.

\smallskip
\noindent {\sc Claim.} {\em The function $Z_0$ on
$\Omega(K^\times)$, defined by
$$
Z_0(\omega;\rho) = \int_{X_K\setminus \sigma^{-1}D_K} \Phi_0(y)\,
\omega\big(\rho(y)\big)\, |\sigma^* dx|_K
$$
for $\omega \in \Omega_{(0,\infty)}(K^\times)$, and extended to
$\Omega(K^\times)$ by meromorphic continuation, belongs to the class
$\mathcal {Z}$ of Definition \ref{functionclassreal}.}

\smallskip
\noindent {\em Proof of the claim.} The integral defining
$Z_0(\omega;\rho)$ is locally of the form
$$
{\displaystyle\int\limits_{K^{n}}} \Theta\left( y\right)
\omega\left( c {\displaystyle\prod_{i=1}^r} y_{i}^{M_{i}}\right)
\left( {\displaystyle\prod_{i=1}^n} \left\vert y_{i}\right\vert_K
^{v_{i}-1}\right) \left\vert dy\right\vert _{K}.
$$
This is precisely the local form that Igusa used to prove the
analogue of the claim in \cite[III \S4]{I3}, up to one important
difference. In his case $M_i>0$ as soon as $v_i>1$. In our case it
is possible that $M_i=0$ while $v_i>1$. This is however not a
problem here because, in the polydiscs $B_j, j\in J_0,$ that we
consider, always $M_i>0$ for {\em at least one} $i$, and this is
precisely what is needed in Igusa's argument.

\smallskip
We now relate in the usual way $Z_0(\omega;\rho)$ to $E_0(z)$
through the fibre integral
$$
F_0(t):= \int_{\rho(y)=t} \Phi_0(y) \left|\frac{\sigma^* dx}{d\rho}\right|_K,
$$
for $t\in K^\times$. We have that
 $$Z_0(\omega;\rho) =
\int_{K^\times} |t|_K F_0(t) \omega(t) \frac{|dt|_K}{|t|_K}$$ is the
Mellin transform of $m_K |t|_K F_0(t)$, where $m_K=2$ for
$K=\mathbb{R}$ and $m_K=2\pi$ for $K=\mathbb{C}$. On the other hand,
$E_0(z)=\int_{K^\times} F_0(t) \Psi(z\cdot t) |dt|_K$ is the Fourier
transform of $F_0(t)$.

 Now, since $Z_0(\omega;\rho)$ belongs to $\mathcal {Z}$, we have by Theorems
\ref{Thm_Igusa_ZF} and \ref{Thm_Igusa_2} that the asymptotic
expansion for $E_0(z)$ at infinity is obtained from the asymptotic
expansion of $F_0(t)$ at zero, by computing its Fourier transform
termwise.
%Combining Theorems \ref{Thm_Igusa_ZF} and \ref{Thm_Igusa_2}, we conclude that $F_0 \in \mathcal {G}$ and $E_0
%\in \mathcal {G}^*$, and that $E_0$ satisfies the asymptotic expansion (\ref{AsimE0}) as $|z|_K \to \infty$,
 Here however the
$\lambda\in \Lambda$ are the poles of $Z_0(\omega;\rho)$ (for all
$\omega$). But by the definition of $\Phi_0$ and the explanation in
Remark \ref{altern}, these are precisely the negative poles of
$Z_\Phi(\omega;f/g)$ (for all $\omega$).

(2) As in the classical case, this follows from \cite[I Theorems 4.2
and 4.3]{I3} and Proposition \ref{Prop Igusa}.
\end{proof}

%?? TO DO Example with $M_i=0$ and $v_i>0$ where analogue of claim is not valid. Maybe not obvious ???

\begin{proposition}
\label{PropEinfty} Let $K$ be an $\mathbb{R}$-field; we put $d_K:=1$
if $K=\mathbb{R}$ and $d_K:=1/2$ if $K=\mathbb{C}$. We denote by
$m_{\lambda}$ the order of a pole $\lambda $ of $Z_{\Phi}\left(
\omega;f/g \right) $. Then $E_{\infty}\left( z\right) $ has an
asymptotic expansion of the form
\begin{equation}
E_{\infty}\left( z\right) -C\approx
{\displaystyle\sum\limits_{\lambda >0}}\
{\displaystyle\sum\limits_{m=1}^{m_{\lambda}+\delta_{\lambda}}}
A_{\lambda,m}\left( \frac{z}{\left\vert z\right\vert }\right)
\left\vert z\right\vert _{K}^{\lambda}\left( \ln\left\vert
z\right\vert _{K}\right) ^{m-1}\text{ as }\left\vert z\right\vert
_{K}\rightarrow0, \label{Asim_E_1}
\end{equation}
where $C=E_\infty(0)$ is a constant, $\lambda$ runs through
$d_K\mathbb{Z}_{>0}$ and all the positive poles of $Z_{\Phi}\left(
\omega;f/g \right) $ (for all $\omega$) that are not in
$d_K\mathbb{Z}_{>0}$; furthermore $\delta_{\lambda}=0$ if $\lambda
\notin d_K\mathbb{Z}_{>0}$
 and $\delta_{\lambda}=1$ if $\lambda\in d_K\mathbb{Z}_{>0}$. When $\lambda$ is not a
pole of $Z_{\Phi}\left( \omega;f/g \right) $, we put
$m_{\lambda}=0$.
 Finally, each $A_{\lambda,m}$ is a
smooth function on $\left\{ u\in K^{\times}\mid \left\vert
u\right\vert _{K}=1\right\}$.
\end{proposition}

\begin{proof} Firstly, we note that the positive poles of $Z_{\Phi}\left(
\omega;f/g \right) $ are precisely the poles of
$$
\int_{X_K\setminus \sigma^{-1}D_K} \Phi_\infty(y)\,
\omega\big(\rho(y)\big)\, |\sigma^* dx|_K = \int_{X_K\setminus
\sigma^{-1}D_K} \Phi_\infty(y)\, \omega^{-1}\big((1/\rho)(y)\big) \,
|\sigma^* dx|_K.
$$
 We consider $1/\rho$ as a $K$-valued function on $X_K\setminus
\rho^{-1}\{0\}$, and we apply the intermediate results in the proof
of Proposition \ref{propE0} to it. More precisely, replacing in that
proof $\rho$ by $1/\rho$, $\Phi_0$ by $\Phi_\infty$ and $\omega$ by
$\omega^{-1}$, we derive that the function
$$F^{\#}(t):=  \int_{(1/\rho)(y)=t} \Phi_\infty(y) \left|\frac{\sigma^* dx}{d(1/\rho)}\right|_K,
$$
for $t\in K^\times$, belongs to the class $\mathcal {G}$, and that
it is related by the Mellin transformation to the function
$Z_\infty$ on $\Omega(K^\times)$, defined by
$$
Z_\infty(\omega;1/\rho) = \int_{X_K} \Phi_\infty(y) \,
\omega\big((1/\rho)(y)\big)\, |\sigma^* dx|_K
$$
for $\omega \in \Omega(K^\times)$, which belongs to the class
$\mathcal {Z}$.   In particular, $F^{\#}$  has the termwise differentiable and uniform asymptotic expansion
\begin{equation}
F^{\#}\left( t\right) \approx \sum_{\lambda\in\Lambda}
\sum_{m=1}^{m_\lambda} b_{\lambda,m}\left( \frac{t}{\left\vert
t\right\vert }\right) \left\vert t\right\vert _{K}^{\lambda-1}\left(
\ln\left\vert t\right\vert _{K}\right) ^{m-1}\text{ as }\left\vert
t\right\vert_{K}\rightarrow0,\label{Fsharp}
\end{equation}
where the $b_{\lambda,m}$ are smooth functions on $\left\{ u\in
K^{\times}\mid \left\vert u\right\vert _{K}=1\right\} $, and where
$\{-\lambda \mid \lambda\in \Lambda\}$ is the set of poles of
$Z_\infty(\omega;1/\rho)$ (for all $\omega$). Note that if
$-\lambda$ is a (necessarily negative) pole of
$Z_\infty(\omega;1/\rho)$, it corresponds to the (positive) pole
$\lambda$ of $Z_\infty(\omega^{-1};1/\rho)$, and thus to the
positive pole $\lambda$ of $Z_{\Phi}\left( \omega;f/g \right) $.

Next, we consider the fibre integral
$$F_\infty(u) :=\int_{\rho(y)=u} \Phi_\infty(y) \left|\frac{\sigma^*
dx}{d\rho}\right|_K,
$$
for $u\in K^\times$. Clearly,
\begin{equation}\label{Einfty-Finfty} E_\infty(z)
=\int_{K^\times} \Psi(z\cdot u)\, F_\infty(u)\,|du|_K,
\end{equation}
and hence it is the Fourier transform of $F_\infty(u)$. Also, we
have, with the change of coordinates $t=1/u$, that
$F_\infty(u)=|u|_K^{-2} F^{\#}(1/u)$. Indeed, since $\sigma^* dx$ is
equal to both
$$\frac{\sigma^*dx}{d\rho} \wedge d\rho \qquad\text{and}\qquad \frac{\sigma^*dx}{d(1/\rho)} \wedge
d(1/\rho) = -\frac{1}{\rho^2}\frac{\sigma^*dx}{d(1/\rho)} \wedge
d\rho ,
$$
and the restriction of the Gel'fand-Leray form to any fibre $\rho=u$
is unique, we have
$$
|-\rho^2|_K \left|\frac{\sigma^* dx}{d\rho}\right|_K =
\left|\frac{\sigma^* dx}{d(1/\rho)}\right|_K .
$$
 Since $F^{\#}$ is a
Schwartz function at infinity, we can extend $F_\infty$ to
 a function on $K$ by declaring $F_\infty(0):=0$,
satisfying $\left. \frac{d^{k}}{du^{k}}F_\infty(u)\right\vert
_{u=0}=0$ for any $k$. So

\smallskip

\noindent (i) $F_\infty\in C^{\infty }\left( K\right) $;

\noindent (ii) we have the termwise differentiable and uniform asymptotic expansion
\begin{equation}
\begin{aligned}
F_\infty(u)&\approx |u|_K^{-2}\ {\displaystyle\sum\limits_{\lambda}}
{\displaystyle\sum\limits_{m=1}^{m_{\lambda}}} b_{\lambda,m}\left(
\left( \frac{u}{\left\vert u\right\vert }\right) ^{-1}\right)
\left\vert u\right\vert _{K}^{-\lambda +1}\left( \ln\left\vert
u^{-1}\right\vert _{K}\right)
^{m-1}\\
&\approx
{\displaystyle\sum\limits_{\lambda}}
{\displaystyle\sum\limits_{m=1}^{m_{\lambda}}} \left( -1\right)
^{m-1}b_{\lambda,m}\left( \left( \frac{u}{\left\vert u\right\vert
}\right) ^{-1}\right) \left\vert u\right\vert _{K}^{-\lambda
-1}\left( \ln\left\vert u\right\vert _{K}\right)
^{m-1}\label{Fkey_Exp}
\end{aligned}
\end{equation}
as $\left\vert u\right\vert _{K}\rightarrow+\infty$.

Hence $F_\infty$ belongs to the space $\mathcal {G}^*$ of Theorem
\ref{Thm_Igusa_2}.
%The point is that its Fourier transform is precisely $E_{\infty}\left( z\right)$.
Then, by (\ref{Einfty-Finfty}), Theorem \ref{Thm_Igusa_2} and
Proposition \ref{Prop Igusa}, we obtain

\smallskip

\noindent (i) $E_{\infty}\in C^{\infty }\left( K^{\times}\right) $;

\noindent (ii) $E_{\infty}$ behaves like a Schwartz function of $z$ as
$\left\vert z\right\vert _{K}$ tends to infinity;

\noindent (iii) we have the asymptotic expansion
\begin{equation}
E_{\infty}(z)\approx {\displaystyle\sum\limits_{\lambda}} \text{ }
{\displaystyle\sum\limits_{m=1}^{m_{\lambda}+\delta_{\lambda}}}
A_{\lambda,m}\left( \frac{z}{\left\vert z\right\vert }\right)
\left\vert z\right\vert _{K}^{\lambda}\left( \ln\left\vert
z\right\vert _{K}\right) ^{m-1}\text{ as }\left\vert z\right\vert
_{K}\rightarrow0,\label{E2Exp}
\end{equation}
where $\lambda$ runs over $\Lambda \cup d_K\mathbb{Z}_{\geq 0}$.
More precisely, we have the following. By Proposition \ref{Prop
Igusa}(3), $w_\chi(\lambda+1)$ can only be zero if $\lambda\in
d_K\mathbb{Z}_{\geq 0}$. Hence, if $\lambda \notin
d_K\mathbb{Z}_{\geq 0}$, or if $\lambda \in d_K\mathbb{Z}_{\geq 0}$
and $m\leq m_\lambda$, the coefficients $A_{\lambda,m}$ in
(\ref{E2Exp}) are explicitly determined by the coefficients
$b_{\lambda,m}$ in (\ref{Fkey_Exp}). But for $\lambda \in
d_K\mathbb{Z}_{\geq 0}$, the coefficients $A_{\lambda,m_\lambda+1}$
are not determined by (\ref{Fkey_Exp}), and can be nonzero.

 In particular, $|z|_K^0$ can appear with nonzero
coefficient $A_{0,1}$ in (\ref{E2Exp}). We claim that this
coefficient must be a constant. Indeed, again by Proposition
\ref{Prop Igusa}(3), $w_{\chi_l}(1)=0$ if and only if
$1=1+\frac1{[K:\mathbb{R}]}(|l|+2k)$ for some $k\in \mathbb{Z}_{\geq
0}$. This can only occur when $l=0$. Hence, writing $A_{0,1,\chi_l}$
for the Fourier coefficients of $A_{0,1}$, i.e.,
$$ A_{0,1}= \sum_l A_{0,1,\chi_l} \chi_l ,$$
we obtain from (\ref{relatebtoc}) that $A_{0,1,\chi_l}=0$ for $l\neq
0$. Since $\chi_0$ is the trivial character, we conclude that
$A_{0,1}$ is the constant function $A_{0,1,\chi_0}$.

This is consistent with the fact that we should have a constant in
the expansion of $E_{\infty}(z)$: from its definition it should be
$E_{\infty}(0)=\int_{X_K\setminus \sigma^{-1} D_K}
(\Phi_\infty)(y)\, |\sigma^* dx|_K$. Note that, in the formulation
of the theorem, we put this constant on the left hand side; for this
reason we must consider only $\lambda>0$ in the asymptotic
expansion.
\end{proof}

%ENOUGH EXPLANATION NOW?

\begin{remark} \rm  Related to the last part of the proof above:
according to \cite[II Theorem 3.1]{I3}, $E_{\infty}(0)$ exists if
and only if $F_{\infty}\in L^1(K)$. This last assertion is indeed
true, since the powers of $|u|_K$ in the expansion (\ref{Fkey_Exp})
are (strictly) smaller than $-1$.
\end{remark}

\bigskip
The treatment of oscillatory integrals over $p$-fields is completely
similar.
% We merely state the results, just indicating some differences.

\begin{proposition}
\label{Prop1p} Let $K$ be a $p$-field. We denote by $m_{\lambda}$
the order of a pole $\lambda $ of $Z_\Phi\left( \omega;f/g\right)$.

(1) Then $E_0\left( z\right) $ has an expansion of the form
\begin{equation}
E_0\left( z\right) = {\displaystyle\sum\limits_{\lambda <0}}\
{\displaystyle\sum\limits_{m=1}^{m_{\lambda}}} A_{\lambda,m}\left(
ac\ z\right) \left\vert z\right\vert _{K}^{\lambda}\left(
\ln\left\vert z\right\vert _{K}\right) ^{m-1} \label{Asim_E_1}
\end{equation}
for sufficiently large $|z|_K$, where $\lambda$ runs through all the
poles$\mod 2\pi i / \ln q$ with negative real part of $Z_\Phi\left(
\omega;f/g\right)$ (for all $\omega$), and each $A_{\lambda,m}$ is a
locally constant function on $R_K^\times$.

(2) Writing $\omega=\omega_s\chi(ac)$, we have that
$A_{\lambda,m_\lambda}=0$ if $\chi=1$ and $\lambda= -1\mod 2\pi i /
\ln q$.
\end{proposition}

\begin{proof}  This can be shown using the same arguments as in the proof of
Proposition \ref{propE0}.
 The refinement now follows from \cite[I Theorem 5.3]{I3} and Proposition \ref{Prop Igusa}.
\end{proof}

\begin{proposition}
\label{Prop2p} Let $K$ be a $p$-field. We denote by $m_{\lambda}$
the order of a pole $\lambda $ of $Z_\Phi\left( \omega;f/g\right)$.
Then $E_\infty\left( z\right) $ has an expansion of the form
\begin{equation}
E_{\infty}\left( z\right) -C= {\displaystyle\sum\limits_{\lambda
>0}}\
{\displaystyle\sum\limits_{m=1}^{m_{\lambda}}} A_{\lambda,m}\left(
ac\ z\right) \left\vert z\right\vert _{K}^{\lambda}\left(
\ln\left\vert z\right\vert _{K}\right) ^{m-1} \label{Asim_E_1}
\end{equation}
for sufficiently small $|z|_K$, where $C=E_\infty(0)$ is a constant,
$\lambda$ runs through all the poles$\mod 2\pi i / \ln q$ with
positive real part of $Z_\Phi\left( \omega;f/g\right)$ (for all
$\omega$), and each $A_{\lambda,m}$ is a locally constant function
on $R_K^\times$.
\end{proposition}

\begin{proof}
This can be shown using the same arguments as in the proof of
Proposition \ref{PropEinfty}. For $p$-fields however, Proposition
\ref{Prop Igusa}(3) is \lq easier\rq, saying that
$w_\chi(\lambda+1)=0$ only if $\chi=1$ and $\lambda=0$$\mod 2\pi i /
\ln q$, yielding as only \lq extra term\rq\ the constant
$C=E_\infty(0)$.
\end{proof}

\section{Expansions for oscillatory integrals}
\label{section expansions}

\subsection{Expansion for $E_{\Phi}\left( z;f/g\right) $ as $z$
approaches zero}

We start with an obvious lemma, and show explicitly the result for
$\mathbb{R}$-fields. The proof for $p$-fields is analogous.

\begin{lemma}
\label{Lemma4} With the notation of Subsection \ref{subsection
set-up resolution}, we have that
$$\lim_{z\rightarrow0}E_{1}\left( z\right)
=\lim_{z\rightarrow0}E_{3}\left( z\right)=
{\displaystyle\int\limits_{K^{n}}} \Theta\left( y\right)\left(
{\displaystyle\prod_{i=1}^n} \left\vert y_{i}\right\vert_K
^{v_{i}-1}\right) \left\vert dy\right\vert _{K}.
$$
\end{lemma}

\begin{proof}
This follows from the Lebesgue Dominated Convergence Theorem.
\end{proof}

%ORGANIZE DEFINITION OF $m_\alpha$ and $m_\beta$

\begin{theorem}\label{Esmallz} Let $K$ be an $\mathbb{R}$-field; we put $d_K:=1$ if
$K=\mathbb{R}$ and $d_K:=1/2$ if $K=\mathbb{C}$. We take
$\frac{f}{g}:U\to K$ and $\Phi$ as in Section \ref{section
preliminaries}.
%Then the following assertions hold.
%\smallskip
%\noindent(1)
 If $Z_{\Phi}\left( \omega;f/g\right)$ has a positive
pole for some $\omega$, then
\begin{align*}
E_{\Phi}\left(  z;f/g\right)  -%
{\displaystyle\int\limits_{K^{n}}}
\Phi\left(  x\right)  \left\vert dx\right\vert _{K}  &  \approx\\%
{\displaystyle\sum\limits_{\gamma>0}}
{\displaystyle\sum\limits_{m=1}^{m_{\gamma}+\delta_\gamma}}
e_{\gamma,m}\left( \frac{z}{\left\vert z\right\vert }\right)
\left\vert z\right\vert _{K}^{\gamma}\left( \ln \left\vert
z\right\vert _{K}\right)^{m-1} \text{ as }\left\vert z\right\vert
_{K} & \rightarrow0,
\end{align*}
where $\gamma$ runs trough $d_K\mathbb{Z}_{>0}$ and all the positive
 poles of $Z_{\Phi}\left( \omega;f/g\right) $, for all
$\omega\in\Omega\left( K^{\times}\right)$, that are not in
$d_K\mathbb{Z}_{>0}$, with $m_{\gamma}$ the order of $\gamma$, and
with $\delta_{\gamma}=0$ if $\gamma \notin d_K\mathbb{Z}_{>0}$
 and $\delta_{\gamma}=1$ if $\gamma\in d_K\mathbb{Z}_{>0}$.
 Finally, each $e_{\gamma,m}$ is a smooth
 function on $\left\{ u\in K^{\times}\mid \left\vert u\right\vert
_{K}=1\right\}$.
%
%\smallskip
%\noindent(2) CHECK If $Z_{\bullet}\left( \omega;f/g\right)$ has a positive pole for some $\omega$, then there exists a positive
%constant $C$ such that for all $\Phi$ (in the analytic setting with support in the compact set $\mathcal{C}$) we have
%\[ \left\vert E_{\Phi}\left( z;f/g\right) -
%{\displaystyle\int\limits_{K^{n}}} \Phi\left( x\right) \left\vert
%dx\right\vert _{K}\right\vert \leq C\left\vert z\right\vert
%_{K}^{_{\alpha}}\big|\ln \left\vert z\right\vert _{K}\big|^{\tilde
%m_{\alpha}+\delta_\alpha-1} \text{ as }\left\vert z\right\vert _{K}
%\rightarrow0,
%\]
%where the smallest positive pole $\alpha$ and its order $\tilde
%m_\alpha$ were defined in Definition \ref{defalphabeta} and Remark
%\ref{deforderalphabeta}, respectively.
\end{theorem}

\begin{proof}  By the definitions in Subsections \ref{subsection set-up
resolution} and \ref{subsection auxexpansions}, we have that
 $E_{\Phi}\left( z;f/g\right) $ is the sum of $E_{\infty}\left(  z\right)$ and  a linear
combination of elementary integrals of types $E_{1}\left( z\right)$
 and $E_{3}\left( z\right) $.
% , see (\ref{E_1})-(\ref{E_3}).
By Lemma \ref{Lemma4}, the contributions of integrals of type
$E_{1}\left( z\right) $ and $E_{3}\left( z\right) $ are constants.
The asymptotic expansion of $E_{\infty}\left( z\right)$ was
calculated in Proposition \ref{PropEinfty}. Finally, the sum of all
constants that appear in the calculations equals
$\int_{K^{n}}\Phi\left( x\right) \left\vert dx\right\vert _{K}$, by
the Lebesgue Dominated Convergence Theorem.
%(WHY DO WE NEED THAT THEOREM?)
%
%(2) This follows from (1) by using Theorem \ref{Theorem3}.
%(RATHER THEOREM \ref{Theorem2}?)
%(Note that in (1) the $\gamma$ and $m_\gamma$ depend on $\Phi$.)
\end{proof}

\begin{theorem}\label{Esmallz-p} Let $K$ be a $p$-field. We take
$\frac{f}{g}:U\to K$ and $\Phi$ as in Section \ref{section
preliminaries}. %Then the following assertions hold.
%\smallskip
%\noindent(1)
If $Z_{\Phi}\left( \omega;f/g\right)$ has a pole with
positive real part for some $\omega$, then
\begin{align*}
E_{\Phi}\left( z;f/g\right) - {\displaystyle\int\limits_{K^{n}}}
\Phi\left(  x\right)  \left\vert dx\right\vert _{K}  & =\\
{\displaystyle\sum\limits_{\gamma>0}}\
{\displaystyle\sum\limits_{m=1}^{m_{\gamma}}} e_{\gamma,m}\left( ac
\, z\right) \left\vert z\right\vert _{K}^{\gamma}\left( \ln
\left\vert z\right\vert _{K}\right)^{m-1} & \text{ for sufficiently
small }\left\vert z\right\vert _{K} ,
\end{align*}
where $\gamma$ runs trough all the poles$\mod 2\pi i / \ln q$ with
positive real part of $Z_{\Phi}\left( \omega;f/g\right) $, for all
$\omega\in\Omega\left( K^{\times}\right)$, and with $m_{\gamma}$ the
order of $\gamma$.
 Finally, each $e_{\gamma,m}$ is a locally constant
function on $R_K^\times$.
%
%\smallskip
%\noindent(2) If $Z_{\bullet}\left( \omega;f/g\right)$ has a pole
%with positive real part for some $\omega$, then there exists a
%positive constant $C$ such that for all $\Phi$ (in the analytic
%setting with support in the compact set $\mathcal{C}$) we have
%\[
%\left\vert E_{\Phi}\left( z;f/g\right) -
%{\displaystyle\int\limits_{K^{n}}} \Phi\left( x\right) \left\vert
%dx\right\vert _{K}\right\vert \leq C\left\vert z\right\vert
%_{K}^{_{\alpha}}\big|\ln \left\vert z\right\vert _{K}\big|^{\tilde
%m_{\alpha}-1} \text{ for sufficiently small }\left\vert z\right\vert
%_{K} ,
%\]
%where the smallest positive pole $\alpha$ and its order $\tilde
%m_\alpha$ were defined in Definition \ref{defalphabeta} and Remark
%\ref{deforderalphabeta}, respectively.
\end{theorem}

\subsection{Expansion for $E_{\Phi}\left( z;f/g\right) $ as $z$
approaches infinity}

We start with two useful preliminary facts.

\begin{lemma}\label{finitecriticalvalues} Let $K$ be an $\mathbb{R}$-field or a $p$-field.
Let $h: V\to K$ be a non-constant $K$-analytic function on an open
set $V$ in $K^n$, and $D$ a compact subset of $V$. Then the
restriction of $h$ to a small neighborhood
% the interior
 of $D$ has only finitely many
critical values.
\end{lemma}

\begin{proof}  This is well known to the experts.
For $K=\mathbb{R}$, it follows from an appropriate version of the
curve selection lemma, see e.g. \cite[Corollary 1.6.1]{O}. We sketch
an alternative argument that also works for $p$-fields $K$. Consider
the restriction of $h$ to a small neighborhood of $D$. Because $D$
is compact, the image of the critical set of this restriction is
globally subanalytic in $K$ (we refer to
\cite{D-vdD},\cite{vdD-H-M},\cite{C} for $p$-fields). Since, by
Sard's theorem, this image is also of measure zero, it can only
consist of a finite set of points.

For $K=\mathbb{C}$, the statement follows from the real case by
considering $h$ as a map from $V\subset \mathbb{R}^{2n}$ to
$\mathbb{R}^2$ and by the Cauchy-Riemann equations.
\end{proof}

As before, let $\Theta$ be the characteristic function of a small
polydisc around the origin, or a smooth function supported in the
polydisc
\[
\left\{ y\in K^{n}\mid\left\vert y_{j}\right\vert_K <1\text{ for
}j=1,\ldots ,n\right\} ,
\]
according as $K$ is a $p$-field or an $\mathbb{R}$-field. We
consider a special case of the integral $E_3$ of (\ref{E_3}), namely
\begin{equation}
\begin{aligned}
E'_{3}\left( z\right) :&= {\displaystyle\int\limits_{K^{n}}}
\Theta\left( y\right) \Psi\big( z \cdot (1+y_1)\big) \left(
{\displaystyle\prod_{i=2}^n} \left\vert y_{i}\right\vert_K
^{v_{i}-1}\right)
 \left\vert
dy\right\vert _{K} \\
&= \Psi(z) {\displaystyle\int\limits_{K^{n}}} \Theta\left( y\right)
\Psi( z \cdot y_1) \left( {\displaystyle\prod_{i=2}^n} \left\vert
y_{i}\right\vert_K ^{v_{i}-1}\right)
 \left\vert
dy\right\vert _{K}. \label{E'_3}
 \end{aligned}
\end{equation}
Note that $y_1$ does not occur in the product in the integrand.

\begin{lemma}
\label{Lemma4bis} (1) Let $K$ be an $\mathbb{R}$-field. Then
$E'_{3}\left( z\right) =O\left( \left\vert z\right\vert
_{K}^{-k}\right) $ as $\left\vert z\right\vert _{K}\rightarrow
\infty$, for any positive number $k$.

(2) Let $K$ be a $p$-field. Then $E'_{3}\left( z\right)=0$ for
sufficiently large $\left\vert z\right\vert _{K}$.
\end{lemma}

\noindent {\em Note.} This can be considered as a special case of
(the proof of) Proposition \ref{propE0}(ii) and Proposition
\ref{Prop1p}(ii), but we think that it is appropriate to mention an
explicit elementary proof.

\begin{proof} (1)  If $K=\mathbb{R}$, we can rewrite $E'_{3}\left( z\right) $
by a standard computation as
\[
e^{2\pi \sqrt{-1}z} {\displaystyle\int\limits_{\mathbb{R}}}
\Omega\left( y_{1}\right) e^{2\pi\sqrt{-1} zy_{1}}\left\vert
dy_{1}\right\vert _{\mathbb{R}},
\]
where $\Omega\left( y_{1}\right) $ is a smooth function supported in
$\left\{ y_{1}\in\mathbb{R}\mid\left\vert y_{1}\right\vert
_{\mathbb{R} }<1\right\} $ (by Lebesgue's Dominated Convergence
Theorem). The result follows by using repeated integration by parts.

If $K=\mathbb{C}$, we take $y_{1}=y_{11}+\sqrt{-1}y_{12}$ with
$y_{11},y_{12}\in\mathbb{R}$, and
$z=z_{1}+\sqrt{-1}z_{2}=|z|_{\mathbb{C}}(\widetilde{z}_{1}+\sqrt{-1}\widetilde{z}_{2})$,
with $z_{1},z_{2}\in\mathbb{R}$. Then $E'_{3}\left( z\right) $ can
be rewritten as
\begin{align}
&e^{4\pi \sqrt{-1}z_{1}} {\displaystyle\int\limits_{\mathbb{R}^{2}}}
\Omega\left( y_{11},y_{12}\right) e^{4\pi \sqrt{-1}\left(
y_{11}z_{1}-y_{12}
z_{2}\right)  }\left\vert dy_{11}dy_{12}\right\vert _{\mathbb{R}}\nonumber\\
 =&\ e^{4\pi \sqrt{-1}z_{1}} {\displaystyle\int\limits_{\mathbb{R}^{2}}}
\Omega\left( y_{11},y_{12}\right) e^{4\pi \sqrt{-1}\left\vert
z\right\vert _{\mathbb{C}}\left(
y_{11}\widetilde{z}_{1}-y_{12}\widetilde{z}_{2}\right) }\left\vert
dy_{11}dy_{12}\right\vert _{\mathbb{R}}, \label{osc_integral}
\end{align}
where
%$\Omega\left( y_{11},y_{12}\right) $
$\Omega$ is a smooth function with support in $\left\{ \left(
y_{11},y_{12}\right) \in\mathbb{R}^{2}\mid y_{11}^{2}
+y_{12}^{2}<1\right\} $. Since we are interested in the behavior of
(\ref{osc_integral}) for $\left\vert z\right\vert _{\mathbb{C}}$ big
enough, we may assume that $z_{1}$ or $z_{2}$ is big enough; we
consider for example the case that $z_{1}$ is big enough. Fixing
$\widetilde{z}_{1}$ and $\widetilde{z}_{2}$, we perform the change
of variables
$y_{11}=\frac{u_{1}}{\widetilde{z}_{1}}+\frac{u_{2}\widetilde{z}_{2}%
}{\widetilde{z}_{1}}$, $y_{12}=u_{2}$, reducing the asymptotics of
(\ref{osc_integral}) to the asymptotics of the integrals considered
in the case $K=\mathbb{R}$.

(2) Say $\Theta$ is the characteristic function of $(P^eR_K)^n$.
Then
$$
E'_{3}\left( z\right) = C \Psi(z)\int_{P^eR_K}
 \Psi\big( z \cdot y_1)\big)
 \left\vert
dy_1\right\vert _{K},
$$
with $C$ a constant, and this integral vanishes when $|z|_K >
q^{-e}$.
\end{proof}

\begin{remark}\label{diffform}\rm
%HERE OR ELSEWHERE SOMETHING LIKE:
In Lemma \ref{Lemma4bis}, it is crucial that the integrand in
(\ref{E'_3}) does not contain a factor $\left\vert
y_{1}\right\vert_K ^{v_{1}-1}$ with $v_1 >1$. Indeed, for example
 in the $p$-field case we have that an integral like
$$\int_{\mathbb{Z}_p} \Psi \big( z(1+y)\big) |y|_{\mathbb{Q}_p}^{v-1}
|dy|_{\mathbb{Q}_p}$$
 does not vanish for sufficiently large $|z|_{\mathbb{Q}_p}$. (A
 straightforward computation shows that it equals
$$\Psi(z) \, \frac{1-p^{v-1}}{1-p^{-v}} \,
|z|_{\mathbb{Q}_p}^{-v} $$
 when $|z|_{\mathbb{Q}_p} \geq p$.)
Similarly, when $K=\mathbb{R}$, consider the integral
$$\int_{\mathbb{R}} \Phi(y) \Psi \big( z(1+y)\big) |y|_{\mathbb{R}}^{v-1}
|dy|_{\mathbb{R}}$$
 with $v$ even, $\Phi$ positive and $\Phi(0)>0$. Then in the
 asymptotic expansion of this integral as $\left\vert z\right\vert_{\mathbb{R}}\rightarrow
\infty$, we have that the term $\Psi(z) |z|_{\mathbb{R}}^{-v}$
appears with non-zero coefficient, see e.g. \cite[p. 183]{AVG}. In
particular, the assertion of Lemma \ref{Lemma4bis}(1) cannot be
true.
\end{remark}

%? Example with infinitely many critical values:
%$\frac{x^2+y^2+x^3\sin x}{x^2}$

\bigskip
For an analytic function $f$, the only relevant expansion of
$E_{\Phi}\left( z;f\right) $ is as $|z|_K \to \infty$. It depends on
all critical points of $f$, but typically one assumes that
\begin{equation}\label{critical}
C_f \cap \operatorname{supp} \Phi \subset f^{-1}(0) .
\end{equation}
This assumption is not really restrictive, because, by Lemma
\ref{finitecriticalvalues}, $f$ has only finitely many critical
values on the support of $\Phi$. Hence, one can reduce to the
situation (\ref{critical}) by partitioning the original support of
$\Phi$ and applying a translation. Note however that, when $c\in
K^\times$ is a critical value of $f$ and $C_f \cap
\operatorname{supp} \Phi \subset f^{-1}(c)$, we have that
$$
E_{\Phi}( z;f) =E_{\Phi}( z;c+(f-c))= \Psi(c\cdot z) E_{\Phi}(
z;f-c).
$$
Therefore, in expansions like in Proposition \ref{propE0}, all terms
will contain an additional factor $\Psi(c\cdot z)$.

The explicit expansion of $E_{\Phi}( z;f)$ as $|z|_K \to \infty$,
without any assumption on $C_f \cap \operatorname{supp} \Phi$, and
hence involving factors $\Psi(c\cdot z)$ for all critical values $c$
of $f$, was treated for $p$-fields in \cite[Proposition 2.7]{D-V}
(when $f$ is a polynomial).

\medskip
Now for our object of study $E_{\Phi}\left( z;f/g\right) $, one {\em
cannot} reduce the problem to a case like (\ref{critical}) by
 pulling back the integral $E_{\Phi}( z;f/g)$ to the manifold $X_K$ and by restricting the support of
$\Phi$.
 Indeed, suppose that $f$ and $g$ have at least one common zero in $K^n$, say
$P$ (this is of course the interesting situation). Fixing an adapted
embedded resolution $\sigma$ of $D_K$, it turns out that, as soon as
$P\in \supp \Phi$, one must take into account factors $\Psi(c\cdot
z)$, where $c$ is a critical value of $\rho=(f/g) \circ \sigma$ or
another \lq special\rq\ value of $\rho$, defined below.

We keep using the notation of Subsection \ref{subsection set-up
resolution}. In particular we consider around each $b\in X_K$ local
coordinates $y$ in a small enough polydisc $B$, such that
$|\sigma^*dx|_K$ is locally of the form $\prod_{i=1}^n
|y_i|_K^{v_i-1}|dy|_K$, and $\rho$ is locally in $y$ a monomial, the
inverse of a monomial, or invertible, according as $\rho(b)$ is $0$,
$\infty$, or in $K^\times$.

\begin{definition}\label{specialpoints}\rm We take $\frac{f}{g}:U\to K$ and $\Phi$ as in Section \ref{section
preliminaries}. We fix an adapted embedded resolution $\sigma$ of
$D_K$, for which we use the notation of Subsection \ref{subsection
set-up resolution}.

For a point $b\in \sigma^{-1}D_K$ that is not a critical point of
$\rho$ and for which $c=\rho(b)\in K^\times$, we have in local
coordinates $y$ that $\rho$ is of the form $c+\ell(y)+h(y)$, where
$\ell(y)$ is a non-zero linear form and $h(y)$ consists of higher
order terms. We define $\mathcal {S} \subset X_K$ as the set of {\em
special points} of $\rho$, being the union of the critical points of
$\rho$ and the non-critical points $b$ as above for which $\ell(y)$
is linearly dependent on the $y_i$ with $v_i>1$.
%We denote $\mathcal {V}:= \rho(\mathcal {C}\cap\supp(\sigma^*\Phi)$, the set of {\em special values}.
\end{definition}

\begin{example}\label{examplecritical} \rm
With notation as before, we take $U=K^2$ and $\Phi$ such that
$\Phi(0)\neq 0$. Let
 \begin{align*}
&(1)\quad f=x^2-y^2 \qquad\text{and}\qquad g=x^2, \quad\text{or} \\
&(2)\quad f=x^2+x^3-y^2 \qquad\text{and}\qquad g=x^2.
\end{align*}
In both cases we obtain an adapted embedded resolution
$\sigma:X_K\to K^2$ simply by blowing up at the origin of $K^2$.
Consider the chart $V$ of $X_K$ with coordinates $x_1,y_1$ given by
$x=x_1, y= x_1y_1$. On $V$ the exceptional curve $E_1$ is given by
$x_1=0$ and $\rho= (f/g)\circ \sigma$ is given by
$$
(1)\quad \rho = 1-y_1^2 \qquad\text{or} \qquad (2)\quad \rho =
1+x_1-y_1^2,
$$
respectively. The point $(x_1,y_1)=(0,0)$ is a critical point of
$\rho$ in case (1) and a non-critical special point of $\rho$ in
case (2). One easily verifies that in both cases there are no other
special points.
%, hence we have twice $\mathcal {V} = \{1\}$.
\end{example}

Note that in the classical theory of Igusa, studying expansions of
polynomial/ana\-ly\-tic functions $f$, a non-critical special point
cannot occur. Indeed, with the notation of Definition
\ref{specialpoints}, then all $v_i=1$ for a point $b\in X_K$
satisfying $\rho(b)\in K^\times$.

\begin{lemma}
The set $\rho\big(\mathcal {S}\cap\supp(\sigma^*\Phi)\big)$ is
finite.
\end{lemma}

\begin{proof} In order to simplify notation, we assume for each point $b\in \supp(\sigma^*\Phi)$  that exactly
$y_{k+1},\dots,y_n$ satisfy $v_i>1$, where $k\geq 0$. (This can be
achieved after a permutation of the coordinates $y$ in $B$.) We
consider in $B$ the subset
$$\mathcal {S}_b := \{a\in B\cap \mathcal {S} \mid a \text{ has local coordinates } (a_1,a_2,\dots,a_k,0,\dots,0)\} .
%$$\mathcal {C}_b := \{a\in B\cap \mathcal {C} \mid a = (a_1,a_2,\dots,a_k,0,\dots,0) \text{ with } a_1a_2\cdots a_k \neq 0\} .
$$
The union of all these $\mathcal {S}_b$ contains $\mathcal
{S}\cap\supp(\sigma^*\Phi)$. Hence, by compactness of
$\supp(\sigma^*\Phi)$, it is enough to show that each $\rho(\mathcal
{S}_b)$ is finite.

\smallskip
Note that there are only finitely many points $b\in X_K$ for which
all $v_i>1$, because they correspond to the (finitely many)
intersection points of exactly $n$ of the components $E_j, j\in T,$
for which $E_j \cap \supp(\sigma^*\Phi) \neq \emptyset$. Therefore,
we may take $k\geq 1$ above.

Assume first that $\rho(b)\in K^\times$, and write $\rho$ in the
local coordinates $y$ as $$w(y_1,\dots,y_k) + w'(y_1,\dots,y_n),$$
where $w$ depends only on $y_1,\dots,y_k$ and $w'$ belongs to the
ideal generated by $y_{k+1},\dots,y_n$. (In particular the constant
term of $w$ is precisely $\rho(b)$.) By definition of $\mathcal
{S}$, the points $a=(a_1,a_2,\dots,a_k,0,\dots,0)$ in $\mathcal
{S}_b$ satisfy
$$0=\frac{\partial \rho}{\partial y_i}(a) =
\frac{\partial w}{\partial y_i}(a_1,\dots,a_k)$$
 for all $i=1,\dots, k$. Hence, these are critical points of the function $w$, and then,
by Lemma \ref{finitecriticalvalues}, we conclude that $w(\mathcal
{S}_b) = \rho(\mathcal {S}_b)$ is finite.

When $\rho(b)=0$, the function $\rho$ is monomial in the local
coordinates $y$, and hence $\rho(B\cap\mathcal {S})$ is $\{0\}$ or
empty. Similarly, when $\rho(b)=\infty$, we have that
$\rho(B\cap\mathcal {S})$ is $\{\infty\}$ or empty.
\end{proof}

\begin{definition}\label{specialvalues}\rm We denote
$\mathcal {V} := \rho\big(\mathcal {S}\cap\supp(\sigma^*\Phi)\big)
\setminus \{\infty\}$, the set of {\em special values}.
\end{definition}

Here we discard $\infty$ in order to simplify the notation for the
summation set in the theorems below. In Example
\ref{examplecritical}, we have twice $\mathcal {V} = \{1\}$.

%We first treat the case of $\mathbb{R}$-fields.

\begin{theorem}\label{thmEinfty}
Let $K$ be an $\mathbb{R}$-field. %; we put $d_K:=1$ if $K=\mathbb{R}$ and $d_K:=1/2$ if $K=\mathbb{C}$???.
 We take $\frac{f}{g}:U\to K$ and $\Phi$ as in Section \ref{section preliminaries}.
 % Then the following assertions hold.
 Let $\mathcal {V}$ be the set of special values as in
Definition \ref{specialvalues}.

(1) If $Z_{\Phi}\left( \omega;f/g - c\right)$ has a negative pole
for some $c\in \mathcal {V}$ and some $\omega$, then
$$
E_{\Phi}\left( z;f/g\right) \approx \sum_{c\in \mathcal {V}}
{\displaystyle\sum\limits_{\gamma_c<0}}\
{\displaystyle\sum\limits_{m=1}^{m_{\gamma_c}}}
e_{\gamma_c,m,c}\left( \frac{z}{\left\vert z\right\vert }\right)
\Psi(c\cdot z)
 \left\vert z\right\vert
_{K}^{\gamma_c}\left( \ln \left\vert z\right\vert _{K}\right)^{m-1}
\text{ as }\left\vert z\right\vert _{K} \rightarrow \infty,
$$
where $\gamma_c$ runs trough all the negative poles of
$Z_{\Phi}\left( \omega;f/g-c\right) $, for all
$\omega\in\Omega\left( K^{\times}\right)$, and with $m_{\gamma_c}$
the order of $\gamma_c$. Finally, each $e_{\gamma_c,m,c}$ is a
smooth function on $\{ u\in K^{\times}\mid$ $\left\vert
u\right\vert_{K}=1\}$.

(2) Writing $\omega=\omega_s\chi_l(ac)$, we have that
$e_{\gamma_c,m_{\gamma_c},c}=0$ if $|\gamma_c| \in
 1+\frac{1}{[K:\mathbb{R}]} (|l| +
2\mathbb{Z}_{\geq 0})$.
\end{theorem}

\begin{proof}
We construct an embedded resolution $\tilde \sigma: X'_K \to X_K$ of
$$D'_K:= \sigma^{-1}(D_K) \cup (\cup_{c\in \mathcal
{V}\setminus\{0\}} \rho^{-1}\{c\}), $$ where we assume that $\tilde
\sigma$ is a $K$-analytic isomorphism outside the inverse image of
$D'_K$. Since $c\neq 0,\infty$ above, we have in particular that
$\tilde \sigma$ is an isomorphism on sufficiently small
neighborhoods of $\rho^{-1}\{0\}$ and $\rho^{-1}\{\infty\}$.

Instead of working with $\sigma$ and $\rho$, we now work with the
compositions $\sigma'=\sigma \circ \tilde \sigma: X'_K \to U$ and
$\rho'= \rho \circ \tilde \sigma: X'_K \to \mathbb{P}^1$. We also
use now $y=(y_1,\dots,y_n)$ as local coordinates on $X'_K$.

Again, we study
$$E_{\Phi}\left(
z;f/g\right)= \int_{X'_K\setminus (\sigma')^{-1} D_K}
\big((\sigma')^* \Phi\big)(y)\, \Psi \big(z\cdot \rho'(y)\big)\,
|(\sigma')^* dx|_K$$
 via local contributions, similarly as in Subsection \ref{subsection
auxexpansions}, but now fixing a more refined appropriate
decomposition of $(\sigma')^* \Phi$.

 We cover
the (compact) support of $(\sigma')^* \Phi$ by finitely many
polydiscs $B_j, j\in J,$ as before, making sure that the (compact)
fibres $(\rho')^{-1}\{c\}\cap \supp ((\sigma')^* \Phi)$, for all
$c\in \mathcal {V}\cup\{0,\infty\}$,
%, and $\rho^{-1}\{\infty\}\cap \supp ((\sigma')^* \Phi)$
 are completely covered by certain $B_j, j\in J_c \subset J,$
% and $B_j, j\in J_\infty \subset J,$
 with centre $b_j$
mapped by $\rho$ to $c$.
% and $\infty$, respectively.
(Note that possibly $0 \notin \mathcal {V}$, as in Example
\ref{examplecritical}.)
 We define $\Phi_c$, for $c\in \mathcal
{V}\cup\{0,\infty\}$, as the restriction of $(\sigma')^* \Phi$ to
$\cup_{j\in J_c}B_j$,
% and $\cup_{j\in J_\infty}B_j$, respectively,
 modified by the partition of the unity that induces
the functions $\Theta$ of Subsection \ref{subsection set-up
resolution} on the local charts $B_j$.

Besides the contributions $E_0(z)$ and $E_\infty(z)$ as in
Subsection \ref{subsection auxexpansions}, we now define also for
$c\in \mathcal {V}\setminus\{0\}$ the contribution
$$E_{c}\left(
z\right)= \int_{X'_K\setminus (\sigma')^{-1} D_K} (\Phi_c)(y)\, \Psi
\big(z\cdot \rho'(y)\big)\, |(\sigma')^* dx|_K .$$

Then we have that
 $E_{\Phi}\left( z;f/g\right) $ is the sum of $E_{\infty}\left(  z\right)$, $E_{0}\left(  z\right)$, the $E_{c}\left(  z\right)$ with $c\in \mathcal
{V}\setminus\{0\}$, and of a linear combination of elementary
integrals of type $E_{3}\left( z\right) $ as in (\ref{E_3}) for
which $c=\rho'(b)\notin\mathcal{V}$. We first consider this last
type of contributions.
% Let $c\in K^\times$ be the constant term of $u(y)$. With the notation of Subsection \ref{subsection set-up resolution}, this just means that $c=\rho(b)$.
We rewrite
$$
E_{3}\left( z\right) = \Psi\big( z \cdot c
\big){\displaystyle\int\limits_{K^{n}}} \Theta\left( y\right)
\Psi\big( z \cdot (u(y)-c) \big) \left( {\displaystyle\prod_{i=1}^n}
\left\vert y_{i}\right\vert_K ^{v_{i}-1}\right)
 \left\vert
dy\right\vert _{K}. \label{problemintegral}
$$
 Then, maybe after restricting
the support of $\Theta$, the integral above can be rewritten in the
form (\ref{E'_3}) by a change of coordinates. By Lemma
\ref{Lemma4bis}, there is no contribution to the asymptotic
expansion as $\left\vert z\right\vert _{K}\rightarrow \infty$.

Next, we saw in the proof of Proposition \ref{PropEinfty} that
$E_{\infty}\left( z\right)$ behaves like a Schwartz function as
$\left\vert z\right\vert _{K}\rightarrow \infty$, and hence also
does not contribute to the asymptotic expansion of $E_{\Phi}\left(
z;f/g\right)$. The expansion of $E_{0}\left( z\right)$ was
calculated in Proposition \ref{propE0}.
%\smallskip \noindent {\sc First case: $c\notin\mathcal {V}$}.
% $c$ is not a critical value of $\rho$, and moreover the linear term of $u(y)-c$ is linearly independent of the $y_i$ with $v_i>1$.
%\smallskip \noindent {\sc Second case: $c\in\mathcal {V}\setminus\{0\}$}.
%either $c$ is a critical value of $\rho$, or the linear term of $u(y)-c$ is linearly dependent of the $y_i$ with $v_i>1$.
Finally, we apply (the proof of) Proposition \ref{propE0} to obtain
for $c\in \mathcal {V}\setminus\{0\}$ the similar expansion
$$\Psi(c\cdot z)
{\displaystyle\sum\limits_{\gamma_c<0}}\
{\displaystyle\sum\limits_{m=1}^{m_{\gamma_c}}}
e_{\gamma_c,m,c}\left( \frac{z}{\left\vert z\right\vert }\right)
 \left\vert z\right\vert
_{K}^{\gamma_c}\left( \ln \left\vert z\right\vert _{K}\right)^{m-1}
\text{ as }\left\vert z\right\vert _{K} \rightarrow \infty
$$
of $E_{c}\left( z\right)$, together with the vanishing property in
(2). All together, this yields the expansion in the statement,
summing over all $c\in \mathcal {V}$. Note that indeed, in the
special case where $0 \notin \mathcal {V}$, also $E_{0}\left(
z\right)$ behaves like a Schwartz function as $\left\vert
z\right\vert _{K}\rightarrow \infty$ and does not contribute to the
total asymptotic expansion (by Proposition \ref{propE0}(ii) or Lemma
\ref{Lemma4bis}).
%(2) This follows from (1) by using Theorem \ref{Theorem3}.
%(RATHER THEOREM \ref{Theorem2}?)
\end{proof}

The following similar result in the $p$-field case is proven
analogously, now invoking Propositions \ref{Prop1p} and \ref{Prop2p}
instead of their real versions.

\begin{theorem}\label{thmEinfty-p}
Let $K$ be a $p$-field. %; we put $d_K:=1$ if $K=\mathbb{R}$ and $d_K:=1/2$ if $K=\mathbb{C}$???.
 We take $\frac{f}{g}:U\to K$ and $\Phi$ as in Section \ref{section preliminaries}.
 % Then the following assertions hold.
 Let $\mathcal {V}$ be the set of special values as in
Definition \ref{specialvalues}.

(1) If $Z_{\Phi}\left( \omega;f/g - c\right)$ has a pole with
negative real part for some $c$ and some $\omega$, then
$$
E_{\Phi}\left( z;f/g\right) = \sum_{c\in \mathcal {V}}
{\displaystyle\sum\limits_{\gamma_c<0}}\
{\displaystyle\sum\limits_{m=1}^{m_{\gamma_c}}}
e_{\gamma_c,m,c}\left( ac\ z \right) \Psi(c\cdot z)
 \left\vert z\right\vert
_{K}^{\gamma_c}\left( \ln \left\vert z\right\vert _{K}\right)^{m-1}
$$
for sufficiently large $\left\vert z\right\vert _{K}$, where
$\gamma_c$ runs trough all the poles$\mod 2\pi\sqrt{-1}/\ln q$ with
negative real part of $Z_{\Phi}\left( \omega;f/g-c\right) $, for all
$\omega\in\Omega\left( K^{\times}\right)$, and with $m_{\gamma_c}$
the order of $\gamma_c$. Finally, each $e_{\gamma_c,m,c}$ is a
locally constant function on $R_K^\times$.

(2) Writing $\omega=\omega_s\chi(ac)$, we have that
$e_{\gamma_c,m_{\gamma_c},c}=0$ if $\chi=1$ and $\gamma_c= -1\mod
2\pi i / \ln q$.
\end{theorem}

\begin{remark} \rm
The form of our expansions in Theorems \ref{Esmallz-p} and
\ref{thmEinfty-p} is consistent with the results in \cite{CGH} and
\cite{CL}, where a similar form of such expansions is derived for
$p$ big enough, in a more general setting, but without information
on the powers of $|z|_K$ and $\ln \left\vert z\right\vert _{K}$.
\end{remark}

\subsection{Examples}

\begin{example}\label{examplecritical1}\rm  We continue the study of case (1) of Example
\ref{examplecritical} when $K$ is a $p$-field. In that case the map
$\sigma$ is also an embedded resolution of $\rho^{-1}\{1\}$. The
exceptional curve $E_1$ has datum $(N_1,v_1)=(0,2)$.

\smallskip
(1) The fibre $\rho^{-1}\{0\}$ is the strict transform of
$f^{-1}\{0\}$ and consists of two (disjoint) components with datum
$(1,1)$. Hence the only possible negative real part of a pole of
$Z_{\Phi}\left( \omega;f/g\right) $ is $-1$. Moreover, by Remark
\ref{remarkpoles-p}, such a pole occurs only if $\chi$ is trivial,
and it must be of the form $-1+\frac{2\pi\sqrt{-1}}{\ln q} k, k\in
\mathbb{Z}$.

The fibre $\rho^{-1}\{1\}$ is given (as a set) by $y_1=0$ in the
chart $V$ and it has datum $(2,1)$, with respect to the function
$\rho -1$. Then the only possible negative real part of a pole of
$Z_{\Phi}\left( \omega;f/g-1\right) $ is $-\frac 12$. By Theorem
\ref{Theorem1} and Remark \ref{remarkpoles-p}, it can only occur if
the order of $\chi$ is $1$ or $2$, and it must be of the form
$-\frac 12 +\frac{\pi\sqrt{-1}}{\ln q} k, k\in \mathbb{Z}$. (Clearly
all poles are of order $1$.)

Hence Theorem \ref{thmEinfty-p} predicts the following possible
terms in the expansion of $E_{\Phi}\left( z;f/g\right)$ for
sufficiently large $|z|_K$:
 $\Psi(z) |z|_K^{-1/2}$ and $\Psi(z) |z|_K^{-1/2} (-1)^{ord
(z)}$. Note that there should be no term $|z|_K^{-1}$ because of
part $(2)$ of that theorem.

\medskip
As an illustration, we consider the concrete case $K=\mathbb{Q}_p$
with $p\neq 2$, and $\Phi$ the characteristic function of
$\mathbb{Z}_p^2$. Then, using Theorem \ref{TheoremDenef} or with
elementary calculations, one can verify that
\begin{equation}\label{Zinfirstexample}
Z_{\Phi}\left(\omega;f/g\right)=
\frac{p^{1+s}+p^2(p-2)p^{-s}+p^{2-2s}-2p+1}{(p+1)(p^{1+s}-1)(p^{1-2s}-1)}
\end{equation}
 and
$$Z_{\Phi}\left(\omega;f/g-1\right)=
\frac{(p-1)^2}{(p^{1+2s}-1)(p^{1-2s}-1)}.
 $$
And with some more effort, one can compute the following explicit
expression for $E_{\Phi}\left( z;f/g\right)$ in this case:
$$ E_{\Phi}\left( z;f/g\right)= \frac{p}{p+1} \Psi(z)
\eta_p(-z)\, |z|_{\mathbb{Q}_p}^{-1/2}
$$
 for sufficiently large $|z|_{\mathbb{Q}_p}$. Here
\[
\eta_p(a) =\left\{
\begin{array}
[c]{lll} 1 & \text{if} & ord(a) \text{ is even}
\\
&  &   \\
\left(\frac{a_0}p\right) & \text{if} & ord(a) \text{ is odd and }
p\equiv 1\!\mod 4
\\
&  &   \\
\sqrt{-1} \left(\frac{a_0}p\right) & \text{if} & ord(a) \text{ is
odd and } p\equiv 3\!\mod 4 ,
\end{array}
\right.
\]
where $a=p^{ord(a)}(a_0+a_1p+a_2p^2+\dots)$ and the $a_i \in
\{0,\dots, p-1\}$. At first sight, this may seem different from the
description above, but $\eta_p(-z)
 |z|_{\mathbb{Q}_p}^{-1/2}$ can be written as a combination of
 $|z|_{\mathbb{Q}_p}^{-1/2}$ and
$|z|_{\mathbb{Q}_p}^{-1/2} (-1)^{ord (z)}$.

\medskip

(2) The fibre $\rho^{-1}\{\infty\}$ is the strict transform of
$g^{-1}\{0\}$ and has datum $(-2,1)$. Hence the only possible
positive real part of a pole of $Z_{\Phi}\left( \omega;f/g\right) $
is $\frac 12$. Then Theorem \ref{Esmallz-p} predicts the following
possible terms in the expansion of $E_{\Phi}\left( z;f/g\right)$ for
sufficiently small $|z|_K$:
 $\Psi(z) |z|_K^{1/2}$ and $\Psi(z) |z|_K^{1/2} (-1)^{ord
(z)}$.

In the concrete case above, one can analogously verify that this is
precisely what happens.

%$$1-\frac{1}{p+1}|z|_{\mathbb{Q}_p}^{1/2}$$ when $ord(z)$ is even
\end{example}

\begin{example}\label{examplecritical2}\rm We continue similarly the study of case (2) of Example
\ref{examplecritical} when $K$ is a $p$-field. In that case the
fibre $\rho^{-1}\{1\}$ is given by $x_1 -y_1^2=0$ in the chart $V$
and it has datum $(1,1)$, with respect to the function $\rho -1$. We
need to perform two more blow-ups to obtain an embedded resolution
of $\rho^{-1}\{1\}\cup E_1$. The two extra exceptional curves $E_2$
and $E_3$ have data $(N_2,v_2)=(1,3)$ and $(N_2,v_2)=(2,5)$,
respectively, with respect to the function $\rho -1$.

\smallskip
(1) The fibre $\rho^{-1}\{0\}$ is again the strict transform of
$f^{-1}\{0\}$ and consists now of one component with datum $(1,1)$.
Then again the only possible poles of
 $Z_{\Phi}\left(
\omega;f/g\right) $ with negative real part must be of the form
$-1+\frac{2\pi\sqrt{-1}}{\ln q} k, k\in \mathbb{Z}$, occurring only
if $\chi$ is trivial.

It is well known that $E_2$ does not contribute to poles of
$Z_{\Phi}\left( \omega;f/g-1\right) $ (since it intersects only one
other component); see for example \cite{D0}, \cite{Lo1}. Then, by
Theorem \ref{Theorem1} and Remark \ref{remarkpoles-p},
 the only possible poles with negative real part of
$Z_{\Phi}\left( \omega;f/g-1\right) $ are
$-1+\frac{2\pi\sqrt{-1}}{\ln q} k, k\in \mathbb{Z}$, occurring only
if $\chi$ is trivial, and $-\frac 52 +\frac{\pi\sqrt{-1}}{\ln q} k,
k\in \mathbb{Z}$, occurring only if the order of $\chi$ is $1$ or
$2$. (Clearly all poles are of order $1$.)

Hence Theorem \ref{thmEinfty-p} predicts the following possible
terms in the expansion of $E_{\Phi}\left( z;f/g\right)$ for
sufficiently large $|z|_K$:
 $\Psi(z) |z|_K^{-5/2}$ and $\Psi(z) |z|_K^{-5/2} (-1)^{ord
(z)}$. Note again that there should be no terms $|z|_K^{-1}$ and
$\Psi(z) |z|_K^{-1}$ because of part $(2)$ of that theorem.

\medskip
As an illustration, we consider now the concrete case
$K=\mathbb{Q}_p$ with $p\neq 2$, and $\Phi$ the characteristic
function of $(p\mathbb{Z}_p)^2$. The same computation as in
(\ref{Zinfirstexample}) yields
$$Z_{\Phi}\left(\omega;f/g\right)=
\frac{p^{1+s}+p^2(p-2)p^{-s}+p^{2-2s}-2p+1}{p^2(p+1)(p^{1+s}-1)(p^{1-2s}-1)}
. $$ With a more elaborate computation, using Theorem
\ref{TheoremDenef}, one can verify that
$Z_{\Phi}\left(\omega;f/g-1\right)$ is

\[
%\begin{aligned}
%&Z_{\Phi}\left(\omega;f/g-1\right)= \\ &
\frac{(p-1)\big(-p^{4+2s} +
(p^2-p+1)p^{4+s} + (-p^2+p-1)p^{2-s} +p^4-p^3+p^2-p+1
\big)}{p^2(p^{5+2s}-1)(p^{1+s}-1)(p^{1-2s}-1)}.
%\end{aligned}
\]

%HOPEFULLY COMPUTE EXPLICIT ENOUGH EXPRESSION FOR $E_{\Phi}\left(z;f/g\right)$;

\medskip

(2) The terms in the expansion of $E_{\Phi}\left( z;f/g\right)$ for
sufficiently small $|z|_K$ are as in Example
\ref{examplecritical1}(2).
\end{example}

\subsection{Estimates}

Our main theorems on asymptotic expansions for oscillatory integrals
imply estimates for them, for large and small $|z|_K$, in terms of
the \lq largest negative pole\rq\ $\beta$ and the \lq smallest
positive pole\rq\ $\alpha$ of Definition \ref{defalphabeta},
respectively. In order to formulate them, we first introduce the
orders of these poles.

\begin{definition}\label{deforderalphabeta}\rm
 We take $\frac{f}{g}:U\to K$ and $\Phi$ as in Section \ref{section preliminaries}.

(1) Take $\alpha, \beta$
 as in Definition \ref{defalphabeta}(1), depending on (the support of)  $\Phi$.
%For a fixed test function $\Phi$, we call
%\[
%{\displaystyle\bigcup\limits_{\omega}} \left\{ \text{poles of
%}Z_{\Phi}\left( \omega;f/g\right) \right\}
%\]
%the set of poles of $Z_{\Phi}\left( \bullet;f/g\right) $.
%For the following definitions, recall that in the analytic setting we only consider test functions $\Phi$ with support in a fixed compact set $\mathcal{C}$, see Remark \ref{convention}.
 When
$\alpha\neq+\infty$, we put
\[
\tilde m_{\alpha}:=\max_{\omega}\left\{ \text{order of
}\alpha\text{\ as pole of }Z_{\Phi}\left( \omega;f/g\right)
\right\},
\]
%the {\em order of}\ $\alpha$,
and when $\beta\neq-\infty$, we put
\[
\tilde m_{\beta}:=\max_{\omega}\left\{ \text{order of }\beta\text{\
as pole of }Z_{\Phi}\left( \omega;f/g\right) \right\} .
\]
%the {\em order of}\ $\beta$.

%For $\omega\in\Omega\left( K^{\times}\right) $ fixed, the correspondence $\Phi\rightarrow Z_{\Phi}\left( \omega;f/g\right) $
%defines a continuous linear functional on the space of test functions.

(2) Whenever $T$ is finite, in particular when $f$ and $g$ are
polynomials, we can take $\alpha, \beta$
 as in Definition \ref{defalphabeta}(2), independent of $\Phi$.
% we let also $\Phi$ vary and we call
%\[
%{\displaystyle\bigcup\limits_{\Phi,\omega}} \left\{ \text{poles of
%}Z_{\Phi}\left( \omega;f/g\right) \right\} ,
%\]
%the set of poles of $Z_{\bullet}\left( \bullet;f/g\right) $.
 When
$\alpha\neq+\infty$, we put
\[
\tilde m_{\alpha}:=\max_{\omega,\Phi}\left\{ \text{order of
}\alpha\text{\ as pole of }Z_{\Phi}\left( \omega;f/g\right)
\right\},
\]
and when $\beta\neq-\infty$, we put
\[
\tilde m_{\beta}:=\max_{\omega,\Phi}\left\{ \text{order of
}\beta\text{\ as pole of }Z_{\Phi}\left( \omega;f/g\right) \right\}
.
\]

\noindent In both cases we call these numbers the {\em order} of
$\alpha$ and $\beta$, respectively. Note that they are less than or
equal to $n$.

\end{definition}

Taking into account Theorem \ref{Theorem3}, the estimates for small
and large $|z|_K$ below follow directly from Theorems \ref{Esmallz}
and \ref{Esmallz-p}, and Theorems \ref{thmEinfty} and
\ref{thmEinfty-p}, respectively.

\begin{theorem}\label{estimate0} We take $\frac{f}{g}:U\to K$ and $\Phi$ as in Section \ref{section
preliminaries}.

\smallskip
\noindent(1) Let $K$ be an $\mathbb{R}$-field. We put
$\delta_{\alpha}=0$ if $\alpha \notin
\frac1{[K:\mathbb{R}]}\mathbb{Z}_{>0}$
 and $\delta_{\alpha}=1$ if $\alpha\in \frac1{[K:\mathbb{R}]}\mathbb{Z}_{>0}$.

 If $Z_{\Phi}\left( \omega;f/g\right)$ has a positive
pole for some $\omega$, then there exists a positive constant $C$
such that
\[
\left\vert E_{\Phi}\left( z;f/g\right) -
{\displaystyle\int\limits_{K^{n}}} \Phi\left( x\right) \left\vert
dx\right\vert _{K}\right\vert \leq C\left\vert z\right\vert
_{K}^{_{\alpha}}\big|\ln \left\vert z\right\vert _{K}\big|^{\tilde
m_{\alpha}+\delta_\alpha-1} \text{ as }\left\vert z\right\vert _{K}
\rightarrow0.
\]

\smallskip
\noindent(2) Let $K$ be $p$-field. If $Z_{\Phi}\left(
\omega;f/g\right)$ has a pole with positive real part for some
$\omega$, then there exists a positive constant $C$ such that
\[
\left\vert E_{\Phi}\left( z;f/g\right) -
{\displaystyle\int\limits_{K^{n}}} \Phi\left( x\right) \left\vert
dx\right\vert _{K}\right\vert \leq C\left\vert z\right\vert
_{K}^{_{\alpha}}\big|\ln \left\vert z\right\vert _{K}\big|^{\tilde
m_{\alpha}-1} \text{ for sufficiently small }\left\vert z\right\vert
_{K} .
\]
\end{theorem}

Here in general $\alpha$ and $\tilde m_{\alpha}$ must be interpreted
in the sense of Definitions \ref{defalphabeta}(1) and
\ref{deforderalphabeta}(1), that is, depending on $\Phi$. Then also
$\delta_{\alpha}$ and $C$ depend on $\Phi$. Whenever $T$ is finite,
in particular when $f$ and $g$ are polynomials, one can consider
$\alpha$ and $\tilde m_{\alpha}$ in the sense of Definitions
\ref{defalphabeta}(2) and \ref{deforderalphabeta}(2). Then $C$,
$\alpha$, $\delta_{\alpha}$ and $\tilde m_{\alpha}$ are independent
of $\Phi$.

A similar remark applies to the next theorem.

 \begin{theorem}\label{estimateinfty} We take $\frac{f}{g}:U\to K$ and $\Phi$ as in Section \ref{section
preliminaries}. Denote for $c\in \mathcal {V}$ the corresponding
number $\beta$, associated in Definition \ref{defalphabeta} to $f/g
-c$, by $\beta_c$. We define $\beta_{\mathcal {V}}:= \max_{c\in
\mathcal {V}} \beta_c$ and its associated order $\tilde
m_{\beta_{\mathcal {V}}}$ as in Definition \ref{deforderalphabeta}.

Assume that for some $\omega$ and some $c\in \mathcal {V}$, the zeta
function $Z_{\Phi}\left( \omega;f/g-c\right)$ has a negative pole or
a pole with negative real part, according as $K$ is an
$\mathbb{R}$-field or a $p$-field, for some $\omega$. Then there
exists a positive constant $C$ such that
\[
\left\vert E_{\Phi}\left( z;f/g\right) \right\vert \leq C\left\vert
z\right\vert _{K}^{_{\beta_{\mathcal {V}}}}\big|\ln \left\vert
z\right\vert _{K}\big|^{\tilde m_{\beta_{\mathcal {V}}}-1}\qquad
\text{ as }\left\vert z\right\vert _{K} \rightarrow \infty.
\]
\end{theorem}

\medskip
\begin{remark} \rm
Let $K$ be a $p$-field. In the classical case, i.e., with $g=1$ and
with $\Phi$ the characteristic function of $R_K^n$, the oscillatory
integral $E_{\Phi}\left( z;f/g\right)$ becomes a traditional
exponential sum. This fact is not true for a general $g$. Consider
however the particular case of a non-degenerate Laurent polynomial
$h(x) = \frac{f(x)}{x^m}$, where $f\in R_K [x_1,\dots,x_n]\setminus
R_K$ and $x^m= \prod_{i=1}^n x_i^{m_i}$ (all $m_i\in
\mathbb{Z}_{\geq 0}$), and the associated exponential sum
$$
S_{\ell,u}(h) := q^{-\ell n} \sum_{\bar x \in
\left((R_K/P_K^\ell)^\times\right)^n} \Psi\big(
u\mathfrak{p}^{-\ell} h(\bar x)\big) ,
$$
where $\ell \in \mathbb{Z}_{>0}$ and $u\in R_K^\times$. Then we have
\begin{equation}
S_{\ell,u}(h) = \int_{(R_K^\times)^n} \Psi\big(
u\mathfrak{p}^{-\ell} h(x)\big) |dx|_K . \label{lastformula}
\end{equation}
Indeed, decompose $(R_K^\times)^n$ as
$$(R_K^\times)^n = \coprod_{\bar x \in \left((R_K/P_K^\ell)^\times\right)^n}
\tilde x + \mathfrak{p}^\ell R_K^n ,
$$
where $\tilde x$ is a representative of $\bar x$. Then on each piece
of the decomposition we have for all $y\in R_K^n$ that $h(\tilde x +
\mathfrak{p}^\ell y)$ is of the form
$$
\frac{f(\tilde x + \mathfrak{p}^\ell y)}{(\tilde x +
\mathfrak{p}^\ell y)^m} = \frac{f(\tilde x) + \mathfrak{p}^\ell
A}{\tilde x^m + \mathfrak{p}^\ell B} = \frac{f(\tilde x) +
\mathfrak{p}^\ell A}{\tilde x^m (1+ \mathfrak{p}^\ell B')} =
\frac{f(\tilde x) + \mathfrak{p}^\ell C}{\tilde x^m}=
 \frac{f(\tilde x)}{\tilde x^m} + \mathfrak{p}^\ell D,
$$
where $A,B,B',C,D \in R_K$, and the second and last equalities are
valid because the components of $\tilde x$ are units. Hence
$h(\tilde x + \mathfrak{p}^\ell y)$ is of the form $h(\tilde x) +
\mathfrak{p}^\ell D$. This implies (\ref{lastformula}).

We can apply Theorem \ref{estimateinfty} to the meromorphic function
$h$ on $R_K^n$ and the characteristic function of $(R_K^\times)^n$
to obtain the estimate
\begin{equation}
|S_{\ell,u}(h)| \leq C q^{\ell\beta_{\mathcal {V}}} {\ell}^{\tilde
m_{\beta_{\mathcal {V}}} -1} \qquad
 \text{ for sufficiently large } \ell , \label{lastestimate}
\end{equation}
where $C$ is a positive constant, and $\beta_{\mathcal {V}}$ and
$m_{\beta_{\mathcal {V}}}$ are as in Theorem \ref{estimateinfty}.
Note however that this result can already be obtained using Igusa's
classical method for estimating
 exponential sums
$\operatorname{mod}p^{\ell}$, due to the fact that $h$ is a regular
function on $(R_K^\times)^n$.

The estimation (\ref{lastestimate}) can be considered as a $p$-adic
(or $\operatorname{mod}p^{\ell}$) counterpart of the estimations for
exponential sums attached to Laurent polynomials over finite fields,
due to Denef and Loeser \cite{D-L-1}.
\end{remark}

\end{document}